\documentclass[12pt]{amsart}
\usepackage{amscd,amsmath,amssymb,amsfonts,verbatim}
\usepackage[cmtip, all]{xy}
\usepackage{graphicx}
\usepackage[active]{srcltx}

\newtheorem{thm}{Theorem}[section]
\newtheorem{prop}[thm]{Proposition}
\newtheorem{lem}[thm]{Lemma}
\newtheorem{cor}[thm]{Corollary}

\theoremstyle{definition}

\newtheorem{defn}[thm]{Definition}
\newtheorem{remk}[thm]{Remark}
\newtheorem{remks}[thm]{Remarks}

\newtheorem{exm}[thm]{Example}
\newtheorem{exms}[thm]{Examples}
\newtheorem{notat}[thm]{Notation}

\numberwithin{equation}{section}

{\hfill$\square$\end{defn}}

{\hfill$\square$\end{remk}}

{\hfill$\square$\end{remks}}

{\hfill$\square$\end{exm}}

{\hfill$\square$\end{exms}}

{\hfill$\square$\end{notat}}

\newcommand{\sF}{{\mathcal F}}

\newcommand{\sO}{{\mathcal O}}

\newcommand{\sV}{{\mathcal V}}

\newcommand{\sX}{{\mathcal X}}

\newcommand{\A}{{\mathbb A}}

\newcommand{\G}{{\mathbb G}}

\renewcommand{\P}{{\mathbb P}}
\newcommand{\Q}{{\mathbb Q}}
\newcommand{\R}{{\mathbb R}}

\newcommand{\Z}{{\mathbb Z}}

\newcommand{\CH}{{\rm CH}}

\newcommand{\surj}{\twoheadrightarrow}
\newcommand{\inj}{\hookrightarrow}

\newcommand{\Hom}{{\rm Hom}}

\newcommand{\Spec}{{\rm Spec \,}}

\newcommand{\ds}{{/\kern-3pt/}}

\newcommand{\Tor}{{\operatorname{Tor}}}

\newcommand{\ov}{\overline}
\newcommand{\wt}{\widetilde}

\def \bdG{\mathbf G}
\def \bdK{\mathbf K}
\renewcommand{\hom}{\text{hom}}

\renewcommand{\dim}{\text{\rm dim}}

\newcommand{\tuborg}{\left\{\begin{array}{ll}}
\newcommand{\sluttuborg}{\end{array}\right.}

\begin{document}
\title{Higher K-theory of Toric stacks}
\author{Roy Joshua and Amalendu Krishna}
\thanks{The first author was
supported by a grant from the National Science Foundation. The second author 
was supported by the Swarnajayanti fellowship, Govt. of India, 2011.}
\address{Department of Mathematics, Ohio State University, Columbus, Ohio, 
43210, USA}
\address{School of Mathematics, Tata Institute of Fundamental Research,
Homi Bhabha Road, Colaba, Mumbai, India}
\email{joshua@math.ohio-state.edu}
\email{amal@math.tifr.res.in}
\baselineskip=10pt



\begin{abstract}
In this paper, we develop several techniques for computing the higher 
G-theory and K-theory of quotient stacks. Our main results for 
computing these groups are in terms of spectral sequences.
 We show that these spectral sequences degenerate in the case of
many toric stacks, thereby providing an efficient computation of their
higher K-theory. 
We apply our main results to give explicit description for the 
higher K-theory for many smooth toric stacks. As another application,
we describe the higher K-theory of toric stack bundles over smooth
base schemes. 
\end{abstract}

\maketitle

\section{Introduction}\label{section:Intro}
Toric varieties form 
a good testing ground for verifying many conjectures in algebraic geometry. 
This becomes particularly apparent when one wants to 
understand  cohomology theories for algebraic varieties.
Computations of 
cohomology rings of smooth toric varieties such as the Grothendieck ring of 
vector bundles, the Chow ring and the singular cohomology ring have been 
well-understood for many years.
These computations facilitate  predictions on the structure of various
cohomology rings of a general algebraic variety.

Just like toric varieties, one would like to have a class of algebraic
stacks on which various cohomological problems about stacks can be tested.
The class of {\sl toric stacks}, first introduced and studied in \cite{BCS} by 
Borisov, Chen and Smith, in terms of 
combinatorial data called {\sl stacky fans},
is precisely such a class of algebraic stacks.
These stacks are expected to be the toy models for
understanding cohomology theories of algebraic stacks, a problem which is
still very complicated in general.

Such a point of view probably accounts for the recent flurry of activity
in this area with several groups considering various forms of toric stacks.
(See for example, \cite{FMN}, \cite{Laf}, \cite{GSI} in addition to \cite{BCS}.)
In \cite{FMN}, Fantechi, Mann and Nironi study the structure of toric 
Deligne-Mumford stacks 
in detail. Recently, in \cite{GSI}, Geraschenko and Satriano consider in detail
{\sl toric stacks} which may not
be Deligne-Mumford.  A class of toric stacks of this kind and their cohomology
were earlier considered by Lafforgue \cite{Laf} in the study of geometric
Langlands correspondence. 
All examples of toric stacks studied before, including those in
\cite{BCS}, \cite{FMN} and \cite{Laf} are shown to be special cases of the 
stacks introduced in \cite{GSI}. 
These stacks appear naturally while solving certain moduli problems and 
computation of their cohomological invariants allows us to understand 
these invariants for many moduli spaces.

In \cite{BCS}, Borisov, Chen and Smith computed the rational Chow ring and 
orbifold Chow ring of toric Deligne-Mumford stacks.  
The integral version of this result for certain type of 
toric Deligne-Mumford stacks is due to Jiang and Tseng \cite{JTseng1}, 
and Iwanari \cite{Iwanari}.  Jiang and Tseng also extend some of these
results to certain toric stack bundles in \cite{Jiang} and \cite{JTseng2}.
Borisov and Horja \cite{BH} computed the integral Grothendieck ring 
$\bdK _0(X)$ of
a toric Deligne-Mumford stack $X$. See also  \cite{Sm}.

On the other hand, almost nothing has been worked out till now regarding the 
higher K-groups of toric stacks, even when they are Deligne-Mumford stacks, 
though the higher K-groups of 
toric varieties had been well understood ({\sl cf.} \cite{VV}) and the
higher Chow groups of toric varieties have been computed recently in
\cite{Krishna}. Furthermore, we still do not know how to compute even the
Grothendieck K-theory ring of a general smooth toric stack.

One goal of this paper is to develop general techniques for computing the
(integral) higher K-theory of smooth toric stacks.
In fact, our results apply to a much bigger class of stacks than just toric
stacks. In particular, these results can be used to describe the higher
equivariant K-theory of many spherical varieties.  
Our general results are in terms of spectral sequences which we show 
degenerate in various cases of interest. This 
allows us to give explicit description of the higher K-theory of toric stacks.

As a consequence of this degeneration of 
spectral sequences, we show how one can recover 
(the integral versions of) and generalize all the
previously known computations of the Grothendieck group 
of toric Deligne-Mumford stacks. As further applications of the main results,
we completely describe the (integral) higher K-theory 
of weighted projective spaces. As another application, we give a complete
description of the higher K-theory of
toric stack bundles over a smooth base scheme.

\subsection{Overview of the main results}
The following is an overview of our main results. We shall fix a base
field $k$ throughout this text. 
A {\sl scheme} in this paper will mean a separated and reduced 
scheme of finite type over $k$. 
A {\sl linear algebraic group} $G$ over $k$ will mean a smooth and affine 
group scheme over $k$. By a closed subgroup $H$ of an algebraic group $G$,
we shall mean a morphism $H \to G$ of algebraic groups over $k$ which is a
closed immersion of $k$-schemes. In particular, a closed subgroup of a linear
algebraic group will be of the same type and hence smooth.  
An algebraic group $G$ will be called {\sl diagonalizable} if it is
a product of a split torus over $k$ and a finite abelian group of order
prime to the characteristic of $k$. In particular, we shall be dealing
with only those tori which are split over $k$. Unless mentioned otherwise,
all products of schemes will be taken over $k$.
A $G$-scheme will mean a scheme with an action of the algebraic group $G$.

For a $G$-scheme $X$, let $\bdG ^G(X)$ (resp. $\bdK ^G(X)$) denote the 
spectrum of the 
K-theory of $G$-equivariant coherent sheaves (resp. vector bundles) on $X$. 
Let $R(G)$ denote the representation ring of $G$. This is canonically
identified with $\bdK ^\bdG _0(k)$. If ${\mathfrak X}$ denotes an 
algebraic stack,
we let $\bdK ({\mathfrak X})$ ($\bdG ({\mathfrak X})$) denote the 
Quillen K-theory
(G-theory) of the exact category of vector bundles 
(coherent sheaves, respectively) on
the stack ${\mathfrak X}$. 
For a quotient stack $\mathfrak{X} = [X/G]$, the spectrum $\bdK (\mathfrak{X})$
(resp. $\bdG (\mathfrak{X})$) is canonically weakly equivalent to the 
equivariant
K-theory $\bdK ^G(X)$ (resp. G-theory $\bdG ^G(X)$) of $X$. 
See \S ~\ref{subsection:K-thry} for more details.

Recall (see below) that a {\sl stacky} toric stack $\mathfrak{X}$ is of the 
form $[X/G]$ where $X$ is a toric variety with dense torus $T$ and $G$ is a
diagonalizable group with a given morphism $\phi: G \to T$.

\vskip .3cm

Our first result is the construction of a spectral sequence which allows
one to compute the higher K-theory of the stack $[X/G]$ from 
the K-theory of the stack $[X/T]$,
whenever a torus $T$ acts on a scheme $X$ and $\phi : G \to T$ is 
a morphism of diagonalizable groups. 
This is related to the spectral sequence of Merkurjev 
(\cite[Theorem 5.3]{Merk}), whose $E_2$-terms are expressed in terms of
${\bdG}_*([X/T])$ and which converges to ${\bdG}_*(X)$ (See also \cite{Lev} for
related constructions). 

We also prove the degeneration of our spectral sequences in many cases, which
 provides an efficient tool for computing the higher K-theory of many 
quotient stacks, including toric stacks.

\begin{thm}
\label{thm:main-thm-1}
Let $T$ be a split torus acting on a scheme $X$ and let $\phi: G \to T$
be a morphism of diagonalizable groups so that $G$ acts on $X$ via $\phi$.
Then, there is a spectral sequence:
\begin{equation}\label{eqn:gen.weak.eq1}
E^{s,t}_2 = {\Tor}_{s}^{R(T)}(R(G), \bdG _t([X/T])) \Rightarrow \bdG _{s+t}([X/G]).
\end{equation}
Moreover, the edge map $\bdG _0([X/T]) {\underset {R(T)} \otimes} R(G) \to
\bdG _0([X/G])$ is an isomorphism.

The spectral sequence ~\eqref{eqn:gen.weak.eq1} degenerates at the 
$E_2$-terms if $X$ is a smooth toric variety with dense torus $T$ such that
$\bdK _0([X/T])$ is a projective $R(T)$-module and we obtain 
the ring isomorphism:
\begin{equation}\label{eqn:gen.weak.eq2}
\bdK _*([X/T]) {\underset {R(T)} \otimes} R(G) 
\xrightarrow{\cong} \bdK _*([X/G]). 
\end{equation}
In particular, this isomorphism holds when $X$ is a smooth and projective
toric variety.
\end{thm}

If $\mathfrak{X} = [X/G]$ is a generically stacky toric stack associated to 
the data $\underline{X} = (X, G \xrightarrow{\phi} T)$, 
then the above results apply to
the G-theory and K-theory of $\mathfrak{X}$. We shall apply 
Theorem~\ref{thm:main-thm-1} in Subsection~\ref{subsection:BHR}
to give an explicit presentation of the Grothendieck K-theory ring of
a smooth toric stack. If we specialize to the case of smooth toric 
Deligne-Mumford stacks, this recovers the main result of 
Borisov--Horja \cite{BH}.

Another  useful application of Theorem~\ref{thm:main-thm-1} is that
it tells us how we can read off the $T'$-equivariant G-groups of a $T$-scheme 
$X$ in terms of its  $T$-equivariant G-groups, whenever $T'$ is a 
closed subgroup of $T$. 
The special case of the isomorphism 
$\bdG _0([X/T]) {\underset {R(T)} \otimes} R(G)  
\xrightarrow{\cong} \bdG _0([X/G])$ 
when $G$ is the trivial group and 
$X$ is a smooth toric variety, recovers the main result of \cite{Moreli}.

We should also observe that Theorem~\ref{thm:main-thm-1} 
applies to a bigger class of schemes than just toric varieties.
In particular, one can use them to compute the equivariant $K$-theory
of many spherical varieties. Another special case of the isomorphism
$\bdG _0([X/T]) {\underset {R(T)} \otimes} R(G) \xrightarrow{\cong} 
\bdG _0([X/G])$ 
when $G$ is the trivial group and $X$ is a spherical variety, recovers 
the main result of \cite{Takeda}.

\vskip .3cm 

\begin{thm}\label{thm:main-thm-2}
Let $T$ be a split torus acting on a smooth and projective scheme $X$ 
which is $T$-equivariantly linear ({\sl cf.} Definition~\ref{defn:linear}). 
Let $\phi: G \to T$ be a morphism of diagonalizable groups so that $G$ acts on 
$X$ via $\phi$.
Then the map
\begin{equation}\label{eqn:main-2-0}
\rho: \bdK _0([X/G]){\underset {{\mathbb Z}} \otimes} \bdK _*(k) \cong 
\bdK _0([X/G]){\underset {R(G)} \otimes} \bdK ^\bdG _*(k) \to \bdK _*([X/G])
\end{equation}
is a ring isomorphism.
\end{thm}

It turns out that all smooth and projective spherical varieties 
({\sl cf.} \S~\ref{subsubsection:Spherical}) satisfy  the hypothesis of 
Theorem~\ref{thm:main-thm-2}. When $X$ is a smooth projective toric variety 
and  $G = T$, then the above theorem recovers a
result of Vezzosi--Vistoli (\cite[Theorem~6.9]{VV}).
\vspace*{1cm}

As an illustration of how the spectral sequence ~\eqref{eqn:gen.weak.eq1}
degenerates in the cases not covered by Theorems~\ref{thm:main-thm-1}
and ~\ref{thm:main-thm-2}, we prove the following result which describes
the higher K-theory of toric stack bundles.

\begin{thm}\label{thm:main-thm-3}
Let $B$ be a smooth scheme over a perfect field $k$ 
and let $[X/G]$ be a toric stack where $X$ is
smooth and projective. Let $R_G\left(\bdK _*(B), \Delta \right)$ denote the
Stanley-Reisner algebra ({\sl cf.} Definition~\ref{defn:RING}) over
$\bdK _*(B)$ associated to a closed subgroup $G$ of $T$. 
Let $\pi: \mathfrak{X} \to B$ be
a toric stack bundle with fiber $[X/G]$. Then there is a ring isomorphism
\begin{equation}\label{eqn:vanish4*0}
\Phi_G : R_G\left(\bdK _*(B), \Delta \right) \xrightarrow{\cong} 
\bdK _*(\mathfrak{X}).
\end{equation}
\end{thm}

When $G$ is the trivial group, the Grothendieck group 
$\bdK _0(\mathfrak{X})$, was computed in \cite[Theorem~1.2(iii)]{SU}. When
$[X/G]$ is a Deligne-Mumford stack, a computation of $\bdK _0(\mathfrak{X})$ 
appears in \cite{JTseng2}. 

The focus of this paper was to describe the higher K-theory
of toric stacks. Similar description of the motivic cohomology 
(higher Chow groups) of such stacks will appear in \cite{JK}.

\vskip .3cm

Here is an {\it outline of the paper}. The second section is a review of toric
stacks and their K-theory. 
In \S~\ref{section:ELin}, we define the notion of equivariantly linear
schemes and study their G-theory.  We prove 
Theorem~\ref{thm:main-thm-1} in \S~\ref{section:Gen}, which is the most general result of this paper.
We conclude this section with a detailed description of the Grothendieck 
K-theory ring of general smooth toric stacks.

In \S~\ref{section:Kunneth}, we prove a derived K{\"u}nneth formula 
and prove Theorem~\ref{thm:main-thm-2} as a consequence.
We conclude this section by working out the higher K-theory of (stacky) 
weighted projective spaces. 
We study the K-theory of toric stack bundles over smooth base schemes
in the last two sections and conclude by providing a 
complete determination of these.

\section{A review of toric stacks and their K-theory}
\label{section:T-stacks}
In this section, we review the concept of toric stacks from \cite{GSI} and
set up the notations for the G-theory and K-theory of such stacks. This is done
in some detail for the convenience of the reader.

In what follows, we shall fix a base field $k$ and all schemes and algebraic
groups will be defined over $k$. Let ${\sV}_k$ denote the category of $k$-schemes and let $\sV^S_k$ denote the 
full subcategory of smooth $k$-schemes. If $G$ is an algebraic group over $k$,
we shall denote the category of $G$-schemes with $G$-equivariant maps by 
$\sV_G$. The full subcategory of smooth $G$-schemes will be denoted by 
$\sV^S_G$.

\subsection{Toric stacks}\label{subsection:TStacks-def}
\begin{defn} 
\label{toric.stacks.def}
Let $T$ be a torus and let $X$ be a toric variety with dense torus $T$.
According to \cite{GSI},
a {\sl toric stack} $\mathfrak{X}$ is an Artin stack of the form $[X/G]$ 
where $G$ is a subgroup of $T$.

A {\sl generically stacky toric stack} is an 
Artin stack of the form $[X/G]$ where $G$ is a diagonalizable group with a 
morphism $\phi: G \to T$. In this case, the stack $\mathfrak{X}$ has an open 
substack of the form $[T/G]$ which acts on it. 
The action of $\mathfrak{T} = [T/G]$
on $\mathfrak{X}$ is induced from the torus action on $X$. The stack
$\mathfrak{T}$ is often called the {\sl stacky} dense torus of $\mathfrak{X}$.
A generically stacky toric stack $[X/G]$ as above will often be described by the
data $\underline{X} = (X, G \xrightarrow{\phi} T)$. 
\end{defn}

\begin{exms}
Generically stacky toric stacks arise  naturally while one studies
toric stacks. This is because a toric variety $X$ with dense torus $T$
has many $T$-invariant subvarieties which are toric varieties and whose
dense tori are quotients of $T$. If $Z \subsetneq X$ is such a subvariety,
and $G$ is a diagonalizable subgroup of the torus $T$, 
then $[Z/G]$ is not a toric stack but only a generically stacky toric stack.

A (generically stacky) toric stack $\mathfrak{X}$ is called a 
{\sl toric Deligne-Mumford stack} if
it is a Deligne-Mumford stack after forgetting the toric structure.    
It is called smooth if $X$ is a smooth scheme.
As pointed out in the introduction, Deligne-Mumford toric stacks were
introduced for the first time in \cite{BCS} using the notion of stacky
fans. A geometric description of the stacks considered in
\cite{BCS} was given in \cite{FMN} where many nice properties of such stacks
were proven. It turns out that all these stacks are special cases of the
ones defined above.

One extreme case of a toric stack is when $G$ is the trivial group,  
in which case
$\mathfrak{X}$ is just a toric variety. The other extreme case is when 
$G$ is all of $T$: clearly such toric stacks are Artin. Toric stacks of this 
form were considered
before by Lafforgue \cite{Laf}. In general, a toric stack occupies a place 
between these two extreme cases.
If $\mathfrak{X}$ is a toric Deligne-Mumford
stack, then the stacky torus $\mathfrak{T}$ is of the form $T' \times 
\mathfrak{B}_{\mu}$, where $T'$ is a torus and $\mathfrak{B}_{\mu}$ is the 
classifying stack of a finite abelian group $\mu$.

In general, every generically stacky toric stack can be written in the
form $\mathfrak{X} \times \mathfrak{B}_G$, where $\mathfrak{X}$ is a toric
stack and $\mathfrak{B}_G$ is the classifying stack of a diagonalizable 
group $G$. This decomposition often reduces the study of the
cohomology theories of generically stacky toric stacks
to the study the cohomology theories of toric stacks and 
the classifying stacks of diagonalizable groups.  

The coarse moduli space $\pi: \mathfrak{X} \to \ov{X}$ of a Deligne-Mumford
toric stack is a simplicial toric variety whose dense torus is the moduli
space of $\mathfrak{T}$. Conversely, every simplicial toric variety is the
coarse moduli space of a canonically defined toric Deligne-Mumford
stack ({\sl cf.} \cite[\S~4.2]{FMN}). 
\end{exms}

\subsection{Toric stacks via stacky fans}\label{subsection:Fan}
In \cite{GSI}, Geraschenko and Satriano showed that all (generically
stacky) toric stacks
are obtained from {\sl stacky fans} in much the same way toric varieties
are obtained from fans. They describe in detail the dictionary between
toric stacks and stacky fans.
 
Associated to the toric variety $X$ is a fan $\Sigma$ on the lattice of 
1-parameter subgroups of $T$, $L=\Hom_{\rm gp}(\G_m,T)$ 
(see \cite[\S 1.4]{fulton} or \cite[\S 3.1]{cls}). 
The surjection of tori $T\to T/G$ corresponds to the
homomorphism of lattices of 1-parameter subgroups, 
$\beta\colon L\to N=\Hom_{\rm gp}(\G_m,T/G)$. The dual homomorphism, 
$\beta^*\colon \hom(N,\Z)\to \hom(L,\Z)$, is the induced homomorphism of 
characters. Since $T\to T/G$ is surjective, $\beta^*$ is injective, 
and the image of $\beta$ has finite index. 
Therefore, one may define a \emph{stacky fan} as a pair $(\Sigma,\beta)$, 
where $\Sigma$ is a fan on a lattice $L$, and $\beta\colon L\to N$ is a 
homomorphism to a lattice $N$ such that $\beta(L)$ has finite index in $N$.  
Conversely, any stacky fan $(\Sigma,\beta)$ gives rise to a  
toric stack as follows.

Let $X_\Sigma$ be the toric variety associated to $\Sigma$. The dual of 
$\beta$, $\beta^*\colon N^{\vee} \to L^{\vee}$, induces a homomorphism of tori 
$T_\beta\colon T_L\to T_N$, naturally identifying $\beta$ with the induced map 
on lattices of 1-parameter subgroups. Since $\beta(L)$ is of finite index in 
$N$, $\beta^*$ is injective, so $T_\beta$ is surjective. 
Let $G_\beta=\ker(T_\beta)$. Note that $T_L$ is the torus of $X_\Sigma$ and 
$G_\beta\subseteq T_L$ is a subgroup. If $(\Sigma,\beta)$ is a stacky fan, the 
associated  toric stack $\mathfrak{X}_{\Sigma,\beta}$ is defined to be 
$[X_\Sigma/G_\beta]$, with the torus $T_N=T_L/G_\beta$.

A \emph{generically stacky fan} is a pair $(\Sigma,\beta)$, where $\Sigma$ is 
a fan on a lattice $L$, and $\beta \colon L\to N$ is a homomorphism to a 
finitely generated abelian group.
If $(\Sigma, \beta)$ is a generically stacky fan, the associated 
generically stacky toric stack $\mathfrak{X}_{\Sigma,\beta}$ is defined to be 
$[X_\Sigma/G_\beta]$, where the action of $G_\beta$ on $X_\Sigma$ is induced by 
the homomorphism $G_\beta\to D(L^*)=T_L$.

One can give a more explicit description of $\mathfrak{X}_{\Sigma,\beta}$ 
considered above which will show that it is a generically stacky 
toric stack. 
Let $(\Sigma,\beta\colon L\to N)$ be a generically stacky fan and 
let $C(\beta)$ denote the complex $L \xrightarrow{\beta} N$.
Let
 \[
 \Z^s\xrightarrow Q \Z^r\to N\to 0
 \]
be a presentation of $N$, and let $B\colon L\to \Z^r$ be a lift of $\beta$
(which exists).
One defines the fan $\Sigma'$ on $L\oplus \Z^s$ as follows. Let $\tau$ be the 
cone generated by $e_1,\dots, e_s\in \Z^s$. For each 
$\sigma\in \Sigma$, let $\sigma'$ be the cone spanned by $\sigma$ and $\tau$ 
in $L\oplus \Z^s$. Let $\Sigma'$ be the fan generated by all the $\sigma'$. 
Corresponding to the cone $\tau$, we have the closed subvariety 
$Y\subseteq X_{\Sigma'}$, which is isomorphic to 
$X_\Sigma$ since $\Sigma$ is the \emph{star} (sometimes called the \emph{link}) 
of $\tau$ \cite[Proposition 3.2.7]{cls}. One defines 
\[
\xymatrix@R-2pc @C-1pc{
 \llap{$\beta'=B\oplus Q\colon\,$}L\oplus \Z^s\ar[r] & \Z^r\\
 (l,a)\ar@{|->}[r] & B(l)+Q(a).}
\]
Then $(\Sigma',\beta')$ is a generically stacky fan and we see that 
$\mathfrak{X}_{\Sigma,\beta}\cong [Y/G_{\beta'}]$. Note that $C(\beta')$ is 
quasi-isomorphic to $C(\beta)$, so $G_{\beta'}\cong G_\beta$.

Toric stacks and generically stacky toric stacks arise naturally, 
especially in the solution of certain moduli problems. Any toric variety
naturally gives rise to a toric stack. In fact, it is shown in 
\cite[Theorem~6.1]{GSII} that if $k$ is algebraically closed 
field of characteristic zero, then every Artin stack with a dense open torus
substack is a toric stack under certain fairly general conditions.
We refer the readers to \cite{GSI} where many examples of toric and
generically stacky toric stacks are discussed. 

\vskip .3cm

{\sl In the rest of this paper, a {\sl toric stack} will always mean any 
generically stacky toric stack. A toric stack as in Definition ~\ref{toric.stacks.def} will be
called a reduced toric stack or a toric orbifold}.

\vskip .3cm

\subsection{K-theory of quotient stacks}\label{subsection:K-thry}
Let $G$ be a linear algebraic group acting on a scheme $X$.
The spectrum of the K-theory of $G$-equivariant coherent sheaves 
(resp. vector bundles) on $X$ is denoted by $\bdG ^G(X)$ (resp. $\bdK ^G(X)$).
We will let $\bdK ^G$ denote $\bdK ^G(Spec \, k)$.
The direct sum of the homotopy groups of these spectra are denoted by 
$\bdG ^G_*(X)$ and $\bdK ^G_*(X)$. The latter is a graded ring. The natural
map $\bdK ^G(X) \to \bdG ^G(X)$ is a weak equivalence if $X$ is smooth.
For a quotient stack $\mathfrak{X}$ of the form $[X/G]$, one writes
$\bdK ^G(X)$ and $\bdK (\mathfrak{X})$ interchangeably. The ring 
$\bdK ^G_0(k)$ will be
denoted by $R(G)$. This is same as the representation ring of $G$.

The functor $X \mapsto \bdG ^G(X)$ on $\sV_G$ is covariant for proper maps and
contravariant for flat maps. It also satisfies the localization sequence
and the projection formula.
It satisfies the homotopy invariance property in the sense that if
$f: V \to X$ is a $G$-equivariant vector bundle, then the map 
$f^*: \bdG ^G(X) \to \bdG ^G(V)$ is a weak equivalence.
The functor $X \mapsto \bdK ^G(X)$ on $\sV_G$ is a contravariant
functor with values in commutative graded rings.
For any $G$-equivariant morphism $f: X \to Y$, $\bdG ^G(X)$ is a module
spectrum over the ring spectrum $\bdK ^G(Y)$. In particular, $\bdG ^G_*(X)$ is
an $R(G)$-module. We refer to \cite[\S~1]{Thomason1} 
to verify the above properties.

\section{Equivariant G-theory of linear schemes}\label{section:ELin}
We will prove Theorem~\ref{thm:main-thm-1} as a consequence of a more
general result (Theorem~\ref{thm:main-thm-1*}) on the equivariant G-theory
of schemes with a group action. 
In this section, we study the equivariant G-theory
of a certain class of schemes which we call equivariantly linear. 
Such schemes in
the non-equivariant set-up were earlier considered by Jannsen \cite{Jan}
and Totaro \cite{Totaro1}. The G-theory of such schemes in the
non-equivariant set-up was studied in \cite{J01}.
We  end this section with a proof of Theorem~\ref{thm:main-thm-1*}
for equivariantly linear schemes.

\begin{defn}\label{defn:linear}
Let $G$ be a linear algebraic group over $k$ and let $X \in \sV_G$.
\begin{enumerate}
\item
We will say  $X$ is $G$-equivariantly $0$-linear if it is either empty
or isomorphic to $\Spec({\rm Sym}(V^*))$ where $V$ is a finite-dimensional
rational representation of $G$.
\item
For a positive integer $n$, we will say that $X$ is $G$-equivariantly 
$n$-linear if there exists a family of objects $\{U, Y, Z\}$ in
$\sV_G$ such that $Z \subseteq Y$ is a $G$-invariant closed immersion 
with $U$ its complement, $Z$ and one of the schemes $U$ or $Y$ are
$G$-equivariantly $(n-1)$-linear and $X$ is the other member of
the family $\{U, Y, Z\}$.
\item
We will say that $X$ is {\it $G$-equivariantly linear (or simply, $G$-linear)}
if it is $G$-equivariantly $n$-linear for some $n \ge 0$.
\end{enumerate}
\end{defn}

It is immediate from the above definition that if $G \to G'$ is
a morphism of algebraic groups then every $G'$-equivariantly linear scheme is 
also $G$-equivariantly linear.

\begin{defn}\label{defn:T-CELL}
Let $G$ be a linear algebraic group over $k$. 
A scheme $X \in \sV_G$ is called {\sl $G$-equivariantly cellular} 
(or, $G$-cellular) if there is a filtration
\[
\emptyset = X_{n+1} \subsetneq X_n \subsetneq \cdots \subsetneq X_1
\subsetneq X_0 = X
\]
by $G$-invariant closed subschemes such that each
$X_i \setminus X_{i+1}$ is isomorphic to a rational representation 
$V_i$ of $G$. These representations of $G$ are called the (affine) 
$G$-{\em cells} of $X$.
\end{defn}
It is obvious that a $G$-equivariantly cellular scheme is cellular in the 
usual sense ({\sl cf.} \cite[Example~1.9.1]{Fulton1}).

Before we collect examples of equivariantly linear schemes,
we state the following two elementary results which will be used throughout 
this paper.

\begin{lem}\label{lem:elem}
Let $G$ be a diagonalizable group over $k$ and let $H \subseteq G$ be a 
closed subgroup. Then $H$ is also defined over $k$ and is diagonalizable. 
If $T$ is a split torus over $k$, then
all subtori and quotients of $T$ are defined over $k$ and are split over $k$.
\end{lem}
\begin{proof}
The first statement follows from \cite[Proposition~8.2]{Borel}.
If $T$ is a split torus over $k$, then any of its subgroups is defined over 
$k$ and is split by the first assertion. In particular, all quotients of $T$ 
are defined over $k$. Furthermore, all such quotients are split over $k$
by \cite[Corollary~8.2]{Borel}.
\end{proof}

\begin{lem}\label{lem:Open-orbit}
Let $T$ be a split torus acting on a scheme $X$ with finitely many
orbits. Then:
\begin{enumerate}
\item
Any $T$-orbit in $X$ of minimal dimension is closed.
\item
Any $T$-orbit in $X$ of maximal dimension is open.
\end{enumerate}
\end{lem}
\begin{proof}
The first assertion is well known and can be found in 
\cite[Proposition~1.8]{Borel}. We prove the second assertion.

Let $f:S \to X$ be the inertia group scheme over $X$ for the $T$-action and let
$\gamma: X \to S$ denote the unit section. Then for a point $x \in X$,
the fiber $S_x$ of the map $f$ is the stabilizer subgroup of $x$ and the 
dimension of $S_x$ is its dimension at the point $\gamma(x)$. 
For any $s \ge 0$, let $X_{\le s}$
denote the set of points $x \in X$ such that $\dim(S_x) \le s$.
It follows from Chevalley's theorem ({\sl cf.} \cite[\S~13.1.3]{EGAIV})
that each $X_{\le s}$ is open in $X$ (see also \cite[\S~2.2]{VV}).

Let $U \subseteq X$ denote a $T$-orbit of maximal dimension (say, $d$) and let
$x \in U$. Suppose $s = \dim(S_x) \ge 0$. Notice that all points in a $T$-orbit
have the same stabilizer subgroup because $T$ is abelian.
We claim that there is no point
on $X$ whose stabilizer subgroup has dimension less than $s$.
If there is such a point $y \in X$, then $Ty$ is a $T$-orbit of
$X$ of dimension bigger than the dimension of $U$, contradicting our
choice of $U$. This proves the claim.  

It follows from this claim that 
$X_{\le s}$ is a $T$-invariant open subscheme of 
$X$ which is a disjoint union of its $T$-orbits
such that the stabilizer subgroups of all points of $X_{\le s}$ have
dimension $s$. In particular, all $T$-orbits in  $X_{\le s}$ have
dimension $d$. 
We conclude from the first assertion of the lemma that
all $T$-orbits (of closed points) in $X_{\le s}$ (including $U$) are closed in 
$X_{\le s}$. Since there only finitely many orbits in $X$, the same is true
for $X_{\le s}$. We conclude that $X_{\le s}$ is a finite disjoint
union of its closed orbits. Hence these orbits must also be open in $X_{\le s}$.
In particular, $U$ is open in $X_{\le s}$ and hence in $X$.
\end{proof} 

\begin{remk}\label{remk:Gen-diag}
The reader can verify that Lemma~\ref{lem:Open-orbit} is true for the action
of any diagonalizable group. But we do not need this general case.
\end{remk}

The following result yields many examples of equivariantly linear
schemes.

\begin{prop}\label{prop:lin-elem}
Let $T$ be a split torus over $k$ and let $T'$ be a quotient of $T$. 
Let $T$ act on $T'$ via the quotient map. Then the following hold.
\begin{enumerate}
\item
$T'$ is $T$-linear.
\item
A toric variety with dense torus $T$ is $T$-linear.
\item
A $T$-cellular scheme is $T$-linear. 
\item
If $k$ is algebraically closed, then every $T$-scheme with finitely many
$T$-orbits is $T$-linear.
\end{enumerate}
\end{prop}
\begin{proof}
We first prove $(1)$.
It follows from Lemma~\ref{lem:elem} that $T'$ is a split torus.
Hence, it is enough to show using the remark following the definition
of $T$-linear schemes that a split torus $T$ is 
$T$-linear under the multiplication action.

We can write $T = (\G_m)^n$ and consider $\A^n$ as the toric variety with the 
dense torus $T$ via the coordinate-wise multiplication  
so that the complement of $T$ is the union of the coordinate hyperplanes. 
Since $\A^n$ is $T$-linear, it suffices to show that
the union of the coordinate hyperplanes is $T$-linear.
 
We shall prove by induction on the rank of $T$ that any union
of the coordinate hyperplanes in $\A^n$ is $T$-linear.
If $n =1$, then this is obvious. So let us assume that $n > 1$ and
let $Y$ be a union of some coordinate hyperplanes in $\A^n$.
After permuting the coordinates, we can write $Y$ as
$Y^{n}_{\{1, \cdots, m\}} = H_1 \cup \cdots \cup H_m$ where 
$H_i = \{(x_1, \cdots , x_n) \in \A^n | x_i = 0\}$.
If $m =1$, then $Y^{n}_{\{1\}}$ is $T$-equivariantly $0$-linear. 
So we assume by an induction on $m$ that $Y^{n}_{\{2, \cdots , m\}}$ is
$T$-linear.

Set $U = Y^{n}_{\{1, \cdots, m\}} \setminus Y^{n}_{\{2, \cdots , m\}}$.
Then $U$ is the complement of a union of hyperplanes 
$W^{n-1}_{\{2, \cdots ,m\}}$ in $H_1 \cong \A^{n-1}$. Notice that $T$ acts on
$H_1$ through the product $T_1$ of its last $(n-1)$ factors.
By induction on $n$, we conclude that $W^{n-1}_{\{2, \cdots ,m\}}$
is $T_1$-linear. Since $H_1$ is clearly $T_1$-linear, we conclude that $U$ is 
$T_1$-linear and
hence $T$-linear. Thus we have concluded that
both $Y^{n}_{\{2, \cdots , m\}}$ and $U$ are $T$-linear. It follows from this that 
$Y^{n}_{\{1, \cdots, m\}}$ is $T$-linear too.

The assertion $(2)$ easily follows from $(1)$ and an induction on the
number of $T$-orbits in a toric variety. The assertion (3) is immediate
from the definitions, using an induction on the length of the filtration
of a $T$-cellular scheme. To prove $(4)$, let $X$ be a $T$-scheme with only
finitely many $T$-orbits. It follows from Lemma~\ref{lem:Open-orbit}
that $X$ has an open $T$-orbit $U$. Since $k$ is algebraically closed, such
an open $T$-orbit must be isomorphic to a quotient of $T$.
In particular, it is $T$-linear by the first assertion.
An induction of the number of $T$-orbit implies that $X \setminus U$
is $T$-linear. We conclude that $X$ is also $T$-linear. 
\end{proof}

\subsubsection{Spherical varieties}\label{subsubsection:Spherical}
Recall that if $G$ is a connected reductive group over $k$, then a
normal variety $X \in \sV_G$ is called {\sl spherical} if a Borel subgroup of
$G$ has a dense open orbit in $X$. The spherical varieties constitute 
a  large class of varieties with group actions, including toric varieties,
flag varieties and all symmetric varieties. It is known that a spherical
variety $X$ has only finitely many fixed points for the $T$-action
where $T$ is a maximal torus of $G$ contained in $B$.

It follows from a theorem of Bialynicki-Birula \cite{BB}
(generalized to the case of non-algebraically closed fields by Hesselink
\cite{Hessel}) that if $T$ is a split torus over $k$ and if
$X$ is a smooth projective variety with a $T$-action
such that the fixed point locus $X^T$ is isolated, then $X$ is $T$-equivariantly
cellular. We conclude that a  smooth and projective
spherical variety is $T$-cellular and hence $T$-linear.
We do not know if all spherical varieties are $T$-linear.

\vskip .4cm 

\subsection{Equivariant G-theory of equivariantly linear schemes}
\label{subsection:Equiv-linear}
Recall that if a linear algebraic group $G$ acts on a scheme $X$, then 
the G-theory and K-theory of the quotient stack
$[X/G]$ are same as the equivariant G-theory and K-theory
of $X$ for the action $G$.
We shall use this identification throughout this text
without further mention.

The following result from \cite[\S~1.9]{Thomason1} will be used repeatedly
in this text.

\begin{thm}\label{thm:Morita} 
Let $G$ be a linear algebraic group over $k$ and let $H \subseteq G$ be a
closed subgroup of $G$. Then for any $X \in \sV_H$, the map of
spectra
\[
\bdG ([(X \stackrel{H}{\times} G)/G]) \to \bdG ([X/H]) 
\]
is a weak-equivalence. In particular, the map of spectra
\[
\bdG ([(X \times G/H)/G]) \to  \bdG ([X/H]) 
\]
is a weak-equivalence if $X \in \sV_G$. These are weak equivalences of ring 
spectra if $X$ is smooth.
\end{thm}

\vskip .3cm

Recall that for a stack $\mathfrak{X}$, the K-theory spectrum
$\bdK (\mathfrak{X})$ is a ring spectrum and 
$\bdG (\mathfrak{X})$ is a module spectrum over $\bdK (\mathfrak{X})$. 
In the following results, we make essential use of the 
derived smash products of module spectra over 
ring spectra. This is the derived functor of the smash product of spectra
in their homotopy category. 
We refer to \cite{SS} (see also \cite{EKMM} and 
\cite[\S~3]{J01}) for basic results in this direction. 
In general, if $R$ is a ring spectrum and $M, N$ are module
spectra over $R$, the derived smash product of $M$ and $N$ over $R$
will be denoted by $M{\overset L {\underset {R} \wedge}} N$. 
We shall now prove the following special case of Theorem~\ref{thm:main-thm-1*}.
The proof follows a trick used in \cite[Theorem~4.1]{J01} in
a different context.

\begin{prop}\label{prop:Linear-case}
Let $T$ be a split torus and let $X \in \sV_T$ be $T$-linear. 
Let $\phi: G \to T$ be a morphism of diagonalizable groups
such that $G$ acts on $X$ via $\phi$.  
Then the natural map of spectra
\begin{equation}\label{eqn:gen.weak.eq*0}
\bdK ([{\Spec(k)}/G]) {\overset L {\underset {\bdK ([{\Spec(k)}/T])} \wedge}} 
\bdG ([X/T]) \to \bdG ([X/G])
\end{equation} 
is a weak-equivalence.
\end{prop}
\begin{proof}
We assume that $X$ is $T$-equivariantly $n$-linear for some $n \ge 0$.
We shall prove our result by an ascending induction on $n$.
If $n = 0$, then $X \cong \A^n$ and hence by the homotopy invariance,
we can assume that $X = \Spec(k)$, and the result is immediate in this case.
We now assume that $n > 0$. By the definition of $T$-linearity, there are
two cases to consider: 
\begin{enumerate}
\item
There exists a $T$-invariant closed subscheme $Y$ of $X$ with complement $U$
such that $Y$ and $U$ are $T$-equivariantly $(n-1)$-linear.
\item
There exists a $T$-scheme $Z$ which contains $X$ as a $T$-invariant open
subscheme such that $Z$ and $Y = Z \setminus X$ are $T$-equivariantly 
$(n-1)$-linear.
\end{enumerate}

In the first case, the localization fiber sequence in equivariant G-theory
gives us a commutative diagram of fiber sequences in the homotopy
category of spectra\footnote {For spectra, this is same as a cofiber sequence}:
\[
\xymatrix@C2pc{
{\bdK ^G{\overset L {\underset {\bdK ^T} \wedge}} 
\bdG ([Y/T])} \ar@<1ex>[r] \ar@<-1ex>[d] & 
{\bdK ^G{\overset L {\underset {\bdK ^T} \wedge}} \bdG ([X/T])} 
\ar@<1ex>[r] \ar@<-1ex>[d] & {\bdK ^G{\overset L {\underset {\bdK ^T} \wedge}} 
\bdG ([U/T])} \ar@<-1ex>[d] \\
{\bdG ([Y/G])} \ar@<1ex>[r] & {\bdG ([X/G])} \
\ar@<1ex>[r] &{\bdG ([U/G])}.}
\]
The left and the right vertical maps are weak equivalences by the induction.
We conclude that the middle vertical map is a weak equivalence too.

In the second case, we obtain as before, a commutative diagram of fiber 
sequences in the homotopy category of spectra:
\[
\xymatrix@C2pc{
{\bdK ^G{\overset L {\underset {\bdK ^T} \wedge}} 
\bdG ([Y/T])} \ar@<1ex>[r] \ar@<-1ex>[d] & 
{\bdK ^G{\overset L {\underset {\bdK ^T} \wedge}} \bdG ([Z/T])} 
\ar@<1ex>[r] \ar@<-1ex>[d] & {\bdK ^G{\overset L {\underset {\bdK ^T} \wedge}} 
\bdG ([X/T])} \ar@<-1ex>[d] \\
{\bdG ([Y/G])} \ar@<1ex>[r] & {\bdG ([Z/G])} \
\ar@<1ex>[r] &{\bdG ([X/G])}.}
\]
The first two vertical maps are weak equivalences by induction and hence the
last vertical map must also be a weak equivalence. This completes the
proof of the proposition.
\end{proof}

We end this section with the following (rather technical) 
result which will be used in the proof of Theorem~\ref{thm:main-thm-1*}.
Taking $V$ to be $\Spec(k)$, 
this becomes a special case of what is considered in the last Proposition.

\begin{lem}\label{lem:BC-I}
Let $T$ be a split torus over $k$ and let $T'$ be a quotient of $T$.
Let $T$ act on $T'$ via the quotient map and let it act trivially on an affine
scheme $V$. Consider the scheme $X =  V \times T'$ where $T$ acts diagonally.
Let $\phi: G \to T$ be a morphism of diagonalizable groups such that 
$G$ acts on any $T$-scheme via $\phi$.
Then the map of spectra
\begin{equation}\label{eqn:BC-I0}
\bdK ([{\Spec(k)}/G]) {\overset L {\underset {\bdK ([{\Spec(k)}/T])} \wedge}} 
\bdG ([X/T]) \to \bdG ([X/G])
\end{equation} 
is a weak-equivalence.
\end{lem}
\begin{proof}
Let $H$ denote the image of $G$ in $T'$ under the composite map
$G \xrightarrow{\phi} T \surj T'$ and
let $H' = {T'}/H$. Notice that $T'$ is a split torus by Lemma~\ref{lem:elem}.
Since $T$ (and hence $G$) acts trivially on 
the scheme $V$, it follows that $T$ and $G$ act on $X$ via their quotients 
$T'$ and $H$, respectively. Since $X$ is affine and all the underlying groups
are diagonalizable, it follows from \cite[Lemma~5.6]{Thomason1} that the maps
of spectra
\begin{equation}\label{eqn:BC-I1}
\bdG ([X/{T'}]) {\overset L {\underset {\bdK ^{T'}} \wedge}} \bdK ^T \to \bdG ([X/T]) ;
\end{equation}
\[
\bdG ([X/{H}]) {\overset L {\underset {\bdK ^{H}} \wedge}} \bdK ^G \to \bdG ([X/G]) 
\]
are weak equivalences.
Using the first weak equivalence, we obtain
\begin{equation}\label{eqn:BC-I2}
\begin{array}{lll}
\bdG ([X/T]) {\overset L {\underset {\bdK ^{T}} \wedge}} \bdK ^G & {\cong} &
\left(\bdG ([X/{T'}]) {\overset L {\underset {\bdK ^{T'}} \wedge}} \bdK ^T\right)
{\overset L {\underset {\bdK ^{T}} \wedge}} \bdK ^G \\
& {\cong} &  \bdG ([X/{T'}]) {\overset L {\underset {\bdK ^{T'}} \wedge}} \bdK ^G \\
& {\cong} &  \bdG ([X/{T'}]) {\overset L {\underset {\bdK ^{T'}} \wedge}} 
\left(\bdK ^H {\overset L {\underset {\bdK ^{H}} \wedge}} \bdK ^G\right) \\
& {\cong} & \left( \bdG ([X/{T'}]) {\overset L {\underset {\bdK ^{T'}} \wedge}} 
\bdK ^H\right) {\overset L {\underset {\bdK ^{H}} \wedge}} \bdK ^G.
\end{array}
\end{equation}
On the other hand, we have
\begin{equation}\label{eqn:BC-I3}
\begin{array}{lll}
\bdG ([X/{T'}]) {\overset L {\underset {\bdK ^{T'}} \wedge}} 
\bdK ^H & {\cong}^{1} &
\bdG ([X/{T'}]) {\overset L {\underset {\bdK ^{T'}} \wedge}} 
\bdG ([{H'}/{T'}]) \\
& {\cong}^{2} & \bdG ([(X \times H')/{T'}])\\
& {\cong}^{3} & \bdG ([X/H]),
\end{array}
\end{equation}
where the isomorphisms ${\cong}^{1}$ and ${\cong}^{3}$ follow from 
Theorem~\ref{thm:Morita}. 
The isomorphism ${\cong}^{2}$ follows from Propositions~\ref{prop:lin-elem} and
~\ref{prop:Kunneth-Linear-case}.
Combining ~\eqref{eqn:BC-I1}, ~\eqref{eqn:BC-I2} and  ~\eqref{eqn:BC-I3}, 
we get the weak equivalences
\[
\bdG ([X/T]) {\overset L {\underset {\bdK ^{T}} \wedge}} \bdK ^G 
\cong \bdG ([X/H])  {\overset L {\underset {\bdK ^{H}} \wedge}} \bdK ^G \cong
\bdG ([X/G])
\]
and this proves the lemma.
\end{proof}


\section{G-theory of general toric stacks}\label{section:Gen}
This section is devoted to the determination of the G-theory of a general 
(generically stacky) toric stack.
We prove our main results in a much more general set-up where
the underlying scheme with a $T$-action need not be a toric variety.
 
Our first result is a spectral sequence that computes the 
$G$-equivariant G-theory of a $T$-scheme $X$ in terms of its
$T$-equivariant G-theory and the representation ring of $G$ whenever
there is a morphism of diagonalizable groups $\phi: G \to T$.
When the underlying scheme is assumed to be smooth, 
these conclusions may be stated in terms of K-theory instead of G-theory. 

This result specializes to the case of all (generically stacky) toric stacks
when $X$ is assumed to be a toric variety. 
We conclude this section with an explicit presentation
of the Grothendieck K-theory ring of a 
smooth toric stack which may not necessarily be Deligne-Mumford.

\vskip .3cm

We now prove the following main result of this section and derive
its consequences.

\begin{thm}\label{thm:main-thm-1*}
Let $T$ be a split torus acting on a scheme $X$ 
and let $\phi: G \to T$ be a morphism of diagonalizable groups
such that $G$ acts on $X$ via $\phi$.  
Then the natural map of spectra
\begin{equation}\label{eqn:gen.weak.eq*0*}
\bdK ([{\Spec(k)}/G]) {\overset L {\underset {\bdK ([{\Spec(k)}/T])} \wedge}} 
\bdG ([X/T]) \to \bdG ([X/G])
\end{equation} 
is a weak-equivalence.
In particular, one obtains a spectral sequence:
\begin{equation}\label{eqn:gen.weak.eq*1}
E^2_{s,t} = 
{\Tor}_{s,t}^{\bdK ^T_*(k)}(\bdK ^G_*(k), \bdG _*([X/T])) 
\Rightarrow \bdG _{s+t}([X/G]).
\end{equation}
\end{thm}
\begin{proof}
We shall prove the theorem by the noetherian induction on $T$-schemes.
The statement of the theorem is obvious if $X$ is the empty scheme so that
both sides of ~\eqref{eqn:gen.weak.eq*0*} are contractible.
Suppose $X$ is any $T$-scheme such that ~\eqref{eqn:gen.weak.eq*0*} holds when
$X$ is replaced by all its proper $T$-invariant closed subschemes.
We show that  ~\eqref{eqn:gen.weak.eq*0*} holds for $X$. 
This will prove the theorem.

By Thomason's generic slice theorem \cite[Proposition~4.10]{Thomason2},
there exists a $T$-invariant dense open subset $U \subseteq X$ which is affine.
Moreover, $T$ acts on $U$ via its quotient $T'$ which in turn acts freely on
$U$ with affine geometric quotient $U/T$ such that there is a $T$-equivariant
isomorphism $U \cong (U/T) \times T'$. Here, $T$ acts trivially on $U/T$,
via the quotient map on $T'$ and diagonally on $U$.
The weak equivalence of  ~\eqref{eqn:gen.weak.eq*0*} holds for $U$ by
Lemma~\ref{lem:BC-I}.

We now set $Y = X \setminus U$. Then $Y$ is a proper $T$-invariant closed
subscheme of $X$. The localization sequence induces the commutative diagram
of the fiber sequences in the homotopy category of spectra:
\[
\xymatrix@C2pc{
{\bdK ^G{\overset L {\underset {\bdK ^T} \wedge}} 
\bdG ([Y/T])} \ar@<1ex>[r] \ar@<-1ex>[d] & 
{\bdK ^G{\overset L {\underset {\bdK ^T} \wedge}} \bdG ([X/T])} 
\ar@<1ex>[r] \ar@<-1ex>[d] & {\bdK ^G{\overset L {\underset {\bdK ^T} \wedge}} 
\bdG ([U/T])} \ar@<-1ex>[d] \\
{\bdG ([Y/G])} \ar@<1ex>[r] & {\bdG ([X/G])} \
\ar@<1ex>[r] &{\bdG ([U/G])}.}
\]
We have shown above that the right vertical map is a weak equivalence.
The left vertical map is a weak equivalence by the noetherian induction.
We conclude that the middle vertical map is a weak equivalence too.

The existence of the spectral sequence now follows along 
standard lines (see for example, \cite[Theorem IV.4.1]{EKMM}).
\end{proof}

\vskip .3cm 

{\bf{Proof of Theorem~\ref{thm:main-thm-1}:}}
To obtain the spectral sequence ~\eqref{eqn:gen.weak.eq1}, it is enough to 
identify this spectral sequence with the one in ~\eqref{eqn:gen.weak.eq*1}.

To see this, we recall from \cite[Lemma~5.6]{Thomason1} that the maps
$R(T){\underset {\Z} \otimes} \bdK _*(k) \to \bdK _*([{\Spec(k)}/T])$ and
$R(G){\underset {\Z} \otimes} \bdK _*(k) \to \bdK _*([{\Spec(k)}/G])$
are ring isomorphisms.
Since $R(T)$ and $R(G)$ are flat $\Z$-modules, these isomorphisms can be 
written as
\begin{equation}\label{eqn:main-thm-1*0}
R(T) {\overset L {\underset {\Z} \otimes}} \bdK _*(k) 
\xrightarrow{\cong} \bdK _*([{\Spec(k)}/T]) \ {\rm and} \ \
R(G) {\overset L {\underset {\Z} \otimes}} \bdK _*(k) \xrightarrow{\cong}
\bdK _*([{\Spec(k)}/G]),
\end{equation}
where ${\overset L \otimes}$ denotes the derived tensor product.

Let $M^{\bullet} \xrightarrow{\sim} R(G)$ be a flat resolution of $R(G)$ as an
$R(T)$-module. Since $R(T)$ is a flat $\Z$-module, we see that
$M^{\bullet} \xrightarrow{\sim} R(G)$ is a flat resolution of $R(G)$ also as a
$\Z$-module. In particular, we obtain

\begin{equation}\label{eqn:main-thm-1*1}
\begin{array}{lll}
\bdK ^G_*(k) {\overset L {\underset {\bdK ^T_*(k)} \otimes}} \bdG _*([X/T]) & \cong &
\left(R(G) {\overset L {\underset {\Z} \otimes}} \bdK _*(k) \right)
{\overset L {\underset {R(T) {\underset {\Z}\otimes} \bdK _*(k)} \otimes}} 
\bdG _*([X/T]) \\
& \cong & 
\left(M ^{\bullet}{\overset L {\underset {\Z} \otimes}} \bdK _*(k) \right)
{\overset L {\underset {R(T) {\underset {\Z}\otimes} \bdK _*(k)} \otimes}} 
\bdG _*([X/T]) \\
& {\cong}^{1} & 
\left(M^{\bullet} {\underset {\Z}\otimes} \bdK _*(k) \right)
{\overset L {\underset{R(T) {\underset {\Z}\otimes} \bdK _*(k)} \otimes}} 
\bdG _*([X/T]) \\
& {\cong}^{2} & 
\left(M ^{\bullet} {\underset {\Z}\otimes} \bdK _*(k) \right)
{\underset {R(T) {\underset {\Z}\otimes} \bdK _*(k)} \otimes} \bdG _*([X/T]) \\
& {\cong} & 
M ^{\bullet} {\underset {R(T)} \otimes} \left(R(T) {\underset {\Z} \otimes} \bdK _*(k)\right)
{\underset {R(T) {\underset {\Z}\otimes} \bdK _*(k)} \otimes} \bdG _*([X/T]) \\
& {\cong} & M ^{\bullet} {\underset {R(T)} \otimes}  \bdG _*([X/T]) \\
& {\cong}^{3} & M ^{\bullet}{\overset L {\underset {R(T)} \otimes}}  \bdG _*([X/T]) \\
& {\cong} & R(G) {\overset L {\underset {R(T)} \otimes}}  \bdG _*([X/T]),
\end{array}
\end{equation}
where the isomorphism ${\cong}^{1}$ follows because $M^{\bullet}$ is a  
complex of flat $\Z$-modules,
${\cong}^{2}$ follows because $M^{\bullet} {\underset {\Z}\otimes} \bdK _*(k)$
is a complex of flat $R(T) {\underset {\Z}\otimes} \bdK _*(k)$-modules and the
isomorphism ${\cong}^{3}$ follows because $M^{\bullet}$ is a complex of flat 
$R(T)$-modules.  
Taking the homology groups on the both sides, we obtain
\[
{\Tor}_{s,t}^{\bdK ^T_*(k)}(\bdK ^\bdG _*(k), \bdG _*([X/T])) \cong
{\Tor}^{R(T)}_{s,t}(R(G), \bdG _*([X/T]))
\]
which yields the spectral sequence ~\eqref{eqn:gen.weak.eq1}.
The isomorphism of the edge map $\bdG _0([X/T]) 
{\underset {R(T)} \otimes} R(G) \to
\bdG _0([X/G])$ follows immediately from  ~\eqref{eqn:gen.weak.eq1}
and the fact that the equivariant G-theory spectra appearing 
in Theorem ~\ref{thm:main-thm-1} are all connected (have no negative
homotopy groups).

Let us now assume that $X$ is a smooth toric variety with dense torus $T$ such
that $\bdK _0([X/T])$ is a projective $R(T)$-module. 
In this case, we can identify the G-theory and the K-theory.
To show the degeneration of the spectral sequence ~\eqref{eqn:gen.weak.eq1},
it suffices to show that the map
\begin{equation}\label{eqn:Degen1}
\bdG _*([X/T]) {\underset {R(T)} \otimes} R(G) \to 
\bdG _*([X/T]) {\overset L {\underset {R(T)} \otimes}}  R(G)
\end{equation}
is an isomorphism. However, we have
\[
\begin{array}{lll}
\bdG _*([X/T]) {\overset L {\underset {R(T)} \otimes}}  R(G) & {\cong}^{0} &
\left(\bdK ^T_*(k)  {\underset {R(T)} \otimes} \bdG _0([X/T])\right) 
{\overset L {\underset {R(T)} \otimes}}  R(G) \\
& {\cong}^{1} &
\left(\bdK ^T_*(k)  {\overset L {\underset {R(T)} \otimes}} \bdG _0([X/T])\right) 
{\overset L {\underset {R(T)} \otimes}}  R(G) \\
& {\cong}^{2} &
\left(\bdK _*(k)  {\overset L {\underset {\Z} \otimes}} R(T)\right) 
{\overset L {\underset {R(T)} \otimes}} 
\left(\bdG _0([X/T]) {\overset L {\underset {R(T)} \otimes}}  R(G)\right) \\
& {\cong}^{3} &
\left(\bdK _*(k)  {\overset L {\underset {\Z} \otimes}} R(T)\right) 
{\overset L {\underset {R(T)} \otimes}} 
\left(\bdG _0([X/T]) {\underset {R(T)} \otimes}  R(G)\right) \\
& {\cong}^{4} &
\bdK _*(k)  {\overset L {\underset {\Z} \otimes}}
\left(\bdG _0([X/T]) {\underset {R(T)} \otimes}  R(G)\right) \\
& {\cong}^{5} &
\bdK _*(k)  {\underset {\Z} \otimes}
\left(\bdG _0([X/T]) {\underset {R(T)} \otimes}  R(G)\right) \\
& {\cong}^{6} &
\left(\bdK _*(k)  {\underset {\Z} \otimes} \bdG _0([X/T])\right)  
{\underset {R(T)} \otimes}  R(G) \\
& {\cong}^{7} &
\bdG _*([X/T]) {\underset {R(T)} \otimes}  R(G).
\end{array}
\]
The isomorphism ${\cong}^{0}$ follows from \cite[Proposition~6.4]{VV} in general
and also from Theorem~\ref{thm:main-thm-2} when $X$ is projective.
The isomorphism ${\cong}^{1}$ follows 
because $\bdG _0([X/T])$ is projective $R(T)$-module. 
The isomorphism ${\cong}^{2}$ follows from \cite[Lemma~5.6]{Thomason1}
because $R(T)$ is flat $\Z$-module.
The isomorphism ${\cong}^{3}$ follows again from the projectivity of
$\bdG _0([X/T])$ as an $R(T)$-module. The isomorphisms ${\cong}^{4}$ and 
${\cong}^{6}$ are the
associativity of the ordinary and derived tensor products. 
The isomorphism ${\cong}^{5}$
follows because $R(G)$ is a free $\Z$-module 
and $\bdG _0([X/T]) {\underset {R(T)} \otimes}  R(G)$ is a projective $R(G)$-module
and hence is flat as a $\Z$-module.
The isomorphism ${\cong}^{7}$ follows again from \cite[Proposition~6.4]{VV} in 
general and also from Theorem~\ref{thm:main-thm-2} when $X$ is projective.
This proves ~\eqref{eqn:Degen1}.
The projectivity of $\bdG _0([X/T])$ as $R(T)$-module when
$X$ is a smooth and projective toric variety, 
is shown in \cite[Proposition~6.9]{VV}
(see also Lemma~\ref{lem:linear-P}).
The proof of Theorem~\ref{thm:main-thm-1} is now complete. 
$\hspace*{14.4cm} \hfil\square$

\begin{remk}\label{remk:SS}
The spectral sequence ~\eqref{eqn:gen.weak.eq1} is basically an 
Eilenberg-Moore type spectral sequence. 
A spectral sequence similar to the one in ~\eqref{eqn:gen.weak.eq1}
had been constructed by Merkurjev \cite{Merk} in the special case when $G$ is 
the trivial group. The construction of that spectral sequence is considerably 
more involved. This special case ($G = \{e\}$) of the
above construction yields a completely different and
simpler proof of Merkurjev's theorem in the setting of schemes with the
action of split tori.
\end{remk}

\begin{remk}\label{remk:Baggio}
It was shown by Baggio \cite{Bag} that there are examples of non-projective 
smooth toric varieties $X$ such that $\bdG _0([X/T])$ is a projective 
$R(T)$-module. 
This shows that there are smooth non-projective toric varieties for which the 
spectral sequence in Theorem ~\ref{thm:main-thm-1} degenerates.
In all these cases, one obtains a complete description of the K-theory
of the toric stack $[X/G]$.
We shall see in Section~\ref{subsection:WPS} that there are examples
where the spectral sequence of Theorem ~\ref{thm:main-thm-1} degenerates 
even if $\bdG _0([X/T])$ is not a projective $R(T)$-module. 
\end{remk}

\subsection{Grothendieck group of toric stacks}\label{subsection:BHR}
In \cite{BH}, Borisov and Horja had computed the Grothendieck $K$-theory ring
$\bdK _0([X/G])$ when $[X/G]$ is a smooth toric Deligne-Mumford stack.
Recall from \S~\ref{section:T-stacks} that the dense stacky torus of a 
Deligne-Mumford stack is of the form $T' \times \mathfrak{B}_{\mu}$
where $T'$ is a torus and $\mu$ is a finite abelian group.
The following consequence of Theorem~\ref{thm:main-thm-1} generalizes
the result of \cite{BH} to the case of all smooth toric stacks,
not necessarily Deligne-Mumford. Even in this latter case, we obtain a
simpler proof.

\begin{thm}\label{thm:BH}
Let $\mathfrak{X} = [X/G]$ be a smooth and reduced toric stack 
associated to the data $\underline{X} = (X, G \xrightarrow{\phi} T)$. 
Let $\Delta$ be the fan defining $X$ and let $d$ be the number of rays in
$\Delta$.  
Let $I^G_{\Delta}$ denote the ideal of the Laurent polynomial algebra 
$\Z[t^{\pm 1}_1, \cdots , t^{\pm 1}_d]$ generated by the relations:
\begin{enumerate}
\item
$(t_{j_1}-1) \cdots (t_{j_l}-1), \ 1 \le j_p \le d$
such that the rays $\rho_{j_1}, \cdots , \rho_{j_l}$ do not span a cone of 
$\Delta$. 
\item
$\left(\stackrel{d}{\underset{j = 1}\prod}
(t_j)^{<-\chi, v_j>}\right) - 1, \ \chi \in (T/G)^{\vee}$.
\end{enumerate}
Then there is a ring isomorphism
\begin{equation}\label{eqn:BH0}
\phi: \frac{\Z[t^{\pm 1}_1, \cdots , t^{\pm 1}_d]}{I^G_{\Delta}} \xrightarrow{\cong}
\bdK _0(\mathfrak{X}).
\end{equation}
\end{thm} 
\begin{proof} 
It follows from Theorem~\ref{thm:main-thm-1} that the map
$\bdK _0([X/T]) {\underset {R(T)} \otimes} R(G) \xrightarrow{\cong} \bdK _0([X/G])$
is a ring isomorphism.
Since $G$ is a diagonalizable subgroup of $T$ ($[X/G]$ is reduced), the ring 
$R(G)$ is a quotient of $R(T)$ by the ideal 
$J^G_{\Delta} = \left(\chi - 1, \chi \in (T/G)^{\vee}\right)$
({\sl cf.} Lemma~\ref{lem:Groupring}). 
This implies that 
\begin{equation}\label{eqn:BH1}
\bdK _0([X/G]) \cong \frac{\bdK _0([X/T])}{J^G_{\Delta}\bdK _0([X/T])}.
\end{equation}

If we let $\Delta(1) = \{\rho_1, \cdots , \rho_d\}$, then
for each $1 \le j \le d$, there is a unique $T$-equivariant
line bundle $L_j$ on $X$ which has a $T$-equivariant section 
$s_{j} : X \to L_{j}$ and whose zero locus is the orbit
closure $V_j = \ov{O_{\rho_j}}$. Then every character $\chi \in T^{\vee}$
acts on $\bdK _0([X/T])$ by multiplication with the element
$(\stackrel{d}{\underset{j = 1}\prod} ([L_{j}])^{<\chi, v_j>})$
({\sl cf.} \cite[Proposition~4.3]{SU}).
We conclude that there is a ring isomorphism
\begin{equation}\label{eqn:BH2}
\frac{\bdK _0([X/T])}
{\left(\stackrel{d}{\underset{j = 1}\prod} 
([L^{\vee}_{j}])^{<-\chi, v_j>} - 1, \ \chi \in (T/G)^{\vee} \right)}
\xrightarrow{\cong} \bdK _0([X/G]).
\end{equation} 

If $I^T_{\Delta}$ denotes the ideal of $\Z[t^{\pm 1}_1, \cdots , t^{\pm 1}_d]$ 
generated by the relations (1) above, then  
it follows from \cite[Theorem~6.4]{VV} that there is a ring isomorphism 
\begin{equation}\label{eqn:BH3}
\frac{\Z[t^{\pm 1}_1, \cdots , t^{\pm 1}_d]}{I^T_{\Delta}} \xrightarrow{\cong}
\bdK _0([X/T]). 
\end{equation} 

Setting $\phi(t_j) = [L^{\vee}_{j}]$,  
we obtain the isomorphism ~\eqref{eqn:BH0} by combining ~\eqref{eqn:BH2}
and ~\eqref{eqn:BH3}.
\end{proof}

\begin{remk}\label{remk:non-red}
If $[X/G]$ is not a reduced stack and there is an exact sequence
\[
0 \to H \to G \to F \to 0
\]
where $F = {\rm Im}(\phi)$, then the stack $[X/G]$ is isomorphic
to $[X/F] \times \mathfrak{B}_H$. In this case, one obtains an
isomorphism $\bdK _*([X/G]) \cong \bdK _*([X/F]) {\underset {R(F)} \otimes} R(G)$
({\sl cf.} \cite[Lemma~5.6]{Thomason1}).
In particular, if $H$ is a torus, one obtains
$\bdK _*([X/G]) \cong  \bdK _*([X/F]) {\underset {\Z} \otimes} R(H)$.
Thus, we see that the calculation of the K-theory of a (generically stacky) 
toric stack can be easily reduced to the case of reduced stacks.
\end{remk}

\section{A K{\"u}nneth formula and its consequences}
\label{section:Kunneth}
Our goal in this section is to prove Theorem~\ref{thm:main-thm-2} and
give applications.
We shall deduce this theorem from a K{\"u}nneth spectral sequence for the 
equivariant K-theory for the action of diagonalizable groups.
A similar spectral sequence for topological K-theory was constructed long time ago by
Hodgkin \cite{Hodgkin} and Snaith \cite{Snaith}. 
A spectral sequence of this kind in the non-equivariant setting was
constructed by the first author in \cite[Theorem~4.1]{J01}.

\subsection{K{\"u}nneth formula}\label{subsection:KFormula}
Suppose  that $X$ and $X'$ are schemes acted upon by a linear algebraic group 
$G$. In this case, the flatness of $X$ and $X'$ over $k$ implies that
the spectra $\bdG ([X/G])$ and $\bdG ([{X'}/G])$ are module spectra over
the ring spectrum $\bdK ([{\Spec(k)}/G])$. This flatness also ensures that
the external tensor product of coherent 
$\sO$-modules induces a pairing $\bdG ([X/G]) \wedge \bdG ([{X'}/G]) \to 
\bdG ([(X{ \times}{X'})/G])$, where the action of $G$ on $X{\times}{X'}$ 
is the diagonal action. This pairing is compatible with 
the structure of the above spectra as module spectra over the ring spectrum 
$\bdK ([{\Spec(k)}/G])$ so that one obtains the induced pairing:
\[
p_1^* \wedge p_2^*: \bdG ([X/G]) {\overset L {\underset {\bdK ([{\Spec(k)}/ G])} 
\wedge}} \bdG ([{X'}/G]) \to \bdG ([(X \times {X'})/ G]).
\]
This is a map of ring spectra if $X$ and $X'$ are smooth.

\begin{prop}\label{prop:Kunneth-Linear-case}
Let $T$ be a split torus and let $X, X'$ be in $\sV_T$ such that $X$ is 
$T$-linear. Let $\phi: G \to T$ be a morphism of diagonalizable groups
such that $G$ acts on $X$ and $X'$ via $\phi$.  
Then the natural map of spectra
\begin{equation}\label{eqn:gen.weak.eq*0}
\bdG ([X/G]) {\overset L {\underset {\bdK ([{\Spec(k)}/G])} \wedge}} 
\bdG ([X'/G]) \to \bdG ([(X\times X')/G])
\end{equation} 
is a weak-equivalence.

In particular, there exists a first quadrant spectral sequence 
\begin{equation}\label{eqn:MT2*1}
E^2_{s,t} = {\Tor}^{\bdK ^G_*(k)}_{s,t}(\bdG _*([X/G]), \bdG _*([{X'}/G])) 
\Rightarrow \bdG _{s+t}([(X \times X')/G]).
\end{equation}
\end{prop}
\begin{proof}
We assume that $X$ is $T$-equivariantly $n$-linear for some $n \ge 0$.
This proposition is proved by an ascending induction on $n$, along the same 
lines as the proof of Proposition~\ref{prop:Linear-case}. We sketch the
argument.

If $n = 0$, then $X \cong \A^n$ and hence by the homotopy invariance,
we can assume that $X = \Spec(k)$, and the result is immediate in this case.
We now assume that $n > 0$. By the definition of $T$-linearity, there are
two cases to consider: 
\begin{enumerate}
\item
There exists a $T$-invariant closed subscheme $Y$ of $X$ with complement $U$
such that $Y$ and $U$ are $T$-equivariantly $(n-1)$-linear.
\item
There exists a $T$-scheme $Z$ which contains $X$ as a $T$-invariant open
subscheme such that $Z$ and $Y = Z \setminus X$ are $T$-equivariantly 
$(n-1)$-linear.
\end{enumerate}

In the first case, the localization fiber sequence in equivariant G-theory
gives us a commutative diagram of fiber sequences in the homotopy
category of spectra:
\[
\xymatrix@C.8pc{
{\bdG ([X'/G]){\overset L {\underset {\bdK ^G} \wedge}} 
\bdG ([Y/G])} \ar@<1ex>[r] \ar@<-1ex>[d] & 
{\bdG ([X'/G]){\overset L {\underset {\bdK ^G} \wedge}} \bdG ([X/G])} 
\ar@<1ex>[r] \ar@<-1ex>[d] & {\bdG ([X'/G]){\overset L {\underset {\bdK ^G} \wedge}} 
\bdG ([U/G])} \ar@<-1ex>[d] \\
{\bdG ([(Y \times X')/G])} \ar@<1ex>[r] & {\bdG ([(X \times X')/G])} \
\ar@<1ex>[r] &{\bdG ([(U \times X')/G])}.}
\]
The left and the right vertical maps are weak equivalences by the induction
on $n$.
We conclude that the middle vertical map is a weak equivalence.
The second case is proved in the same way where we now use induction
on $Y$ and $Z$ (see the proof of Proposition~\ref{prop:Linear-case}).

The existence of the spectral sequence now follows along 
standard lines (see for example, \cite[Theorem IV.4.1]{EKMM}). 
\end{proof}

\begin{remk}\label{remk:subtorus-case}
As an application of Proposition~\ref{prop:Kunneth-Linear-case},
one can obtain another proof of the special case of the
spectral sequence ~\eqref{eqn:gen.weak.eq1} when $G$ is a closed subgroup of 
$T$. This is done by taking $G = T$, $X' = T/G$ in  ~\eqref{eqn:MT2*1}
and using the Morita weak equivalences $\bdG ([{X'}/T]) \cong \bdG ([\Spec(k)/G])$ and
$\bdG ([(X \times X')/T]) \cong \bdG ([X/G])$. Notice that $X' = T/G$ is $T$-linear
by Proposition~\ref{prop:lin-elem}.
\end{remk}

\begin{cor}[K{\"u}nneth decomposition]\label{cor:Kunneth-dec}
Let $T$ be a split torus over $k$ and let $X$ be a $T$-linear scheme. 
Then the class of the diagonal 
$[\Delta] \in \bdG _0([(X \times X)/G])$ admits
a strong K{\"u}nneth decomposition, i.e., 
may be written as $\stackrel{n}{\underset{i =1}\Sigma}
 p_1^*(\alpha_i) \otimes p_2^*(\beta_i)$, where 
$\alpha_i, \beta_i \in \bdG _0([X/G])$.
\end{cor}
\begin{proof} The spectral sequence of 
Proposition~\ref{prop:Kunneth-Linear-case} shows in general that 
\begin{equation}\label{eqn:Ku-0}
\bdG _0([(X \times X') / G]) 
\cong \bdG _0([X/G]){\underset {R(G)} \otimes} \bdG _0([{X'}/G]).
\end{equation}
The K{\"u}nneth decomposition now follows by taking $X = X'$.
\end{proof}

\vskip .3cm 

{\bf{Proof of Theorem~\ref{thm:main-thm-2}:}}
Let $X$ be a smooth and projective $T$-linear scheme.
Since the group $G$ is diagonalizable, we apply \cite[Lemma~5.6]{Thomason1}
to obtain the isomorphism:
\begin{equation}\label{eqn:main-2*0}
R(G) {\underset {\Z}\otimes} \bdK _*(k) \xrightarrow{\cong} \bdK ^G_*(k)
\end{equation} 
and this provides the first isomorphism of ~\eqref{eqn:main-2-0}.
Since $X$ is smooth, we can identify $G_*([X/G])$ with $\bdK _*([X/G])$.

Let $[x] \in \bdK _*([X/G])$. Then $[x]= p_{1*} (\Delta \circ p_2^*([x]))$. 
Now we use the K{\"u}nneth decomposition for $\Delta$ obtained 
in Corollary~\ref{cor:Kunneth-dec}  and the projection formula 
(since $X$ is projective) to identify the last term with 
$\stackrel{n}{\underset{i =1}\sum} 
\alpha _i \circ p_{1*}p_2^*(\beta _i\circ [x])$. 
The Cartesian square
\[
\xymatrix@C2pc{
X\times X \ar[r]^>>>>>>{p_2} \ar[d]_{p_1} & X \ar[d]^{p'_1} \\
X \ar[r]_<<<<<<{p'_2} & \Spec(k)}
\]
and the flat base-change for the equivariant G-theory 
show that $p_{1*}( p_2^*(\beta_i \circ [x]))$ identifies  
with ${p_{2}'}^*{p_1'}_*(\beta _i \circ [x])$ so that
\begin{equation}\label{eqn:surj.diag}
[x]= \stackrel{n}{\underset{i = 1}\sum} 
\alpha _i \circ {p_2'}^*({p_1'}_*(\beta _i \circ [x])).
\end{equation} 
The class  ${p_1'}_*(\beta _i \circ [x]) \in \bdG ([{\Spec(k)}/G])$. 
It follows that the classes $\{\alpha_i \}$ generate $\pi_*(G[X/G])$ as a 
module over $\bdK ^G_*(k)$. This shows that the map in question is surjective.

Next we prove the injectivity of the map $\rho$. 
The key is the following diagram:
\begin{equation}\label{eqn:inj.dig}
\xymatrix@C2pc{
\bdK _*([X/G]) \ar@<1ex>[dr]_{\mu} &
{\bdK _0([X/G]) {\underset {\bdK ^G_0(k)} \otimes} \bdK ^G_*(k)} 
\ar[l]_<<<<<<{\rho} \ar@<1ex>^{\alpha}[d] \\
& {{\Hom}_{\bdK ^G_0(k)}(\bdK _0([X/G]), \bdK ^G_*(k))}}
\end{equation}
where $\alpha(x \otimes y)$ (resp. $\mu(x)$, $ x \in \bdK _*([X/G])$) 
is defined by $ \alpha (x \otimes y) =$ the map 
$x' \mapsto f_*(x' \circ x) \circ y$ (resp., the map $x' \mapsto
f_*(x' \circ x)$). 
Here, $f$ denotes the projection map $X \to \Spec(k)$ and
$x' \circ x$ denotes the product in the ring $\bdK _*([X/G])$.
The commutativity of the above diagram is
an immediate consequence of the projection formula: observe that $\rho(x
\otimes y) = x \circ f^*(y)$. Therefore, to show that 
$ \rho$ is injective, it suffices to show that the map $\alpha$ is
injective.
For this, we define a map $\beta$ to be a splitting for
$\alpha$ as follows. 

If $\phi \in {\Hom}_{\bdK ^G_0(k)}(\bdK _0([X/G]), \bdK ^G_*(k))$,
we let
$\beta(\phi) = \stackrel{n}{\underset{i =1}\sum} 
\alpha_{i} \otimes (\phi(\beta_{i}))$.
Observe that 
\[
\begin{array}{lll}
\beta (\alpha (x \otimes y)) & = &
\beta \left(the \quad map
\quad {x'} \rightarrow f_*(x' \circ x) \circ y\right) \\
& = & (\stackrel{n}{\underset{i =1}\sum} 
\alpha_{i} \otimes f_*(\beta_{i} \cdot x)) \circ y.
\end{array}
\]

We next observe that $f_*( \beta_{i} \cdot x) \in \bdK ^G_0(k)$, so that we may
write the last term as $(\stackrel{n}{\underset{i =1}\sum}  
\alpha_{i} . f^*f_*( \beta_{i} \cdot x)) \circ y$. 
By ~\eqref{eqn:surj.diag}, the last term $=x \circ y$.
This proves that $ \alpha$ is injective and hence that
so is $\rho$. This completes the proof. 
$\hspace*{3cm} \hfil \square$

\vskip .4cm 

The following result generalizes ~\eqref{eqn:gen.weak.eq2} to a bigger
class of schemes.

\begin{cor}\label{cor:Base-change}
Let $T$ be a split torus over $k$ and let 
$X$ be a smooth and projective $T$-linear scheme.
Let $\phi: G \to T$ be a morphism of diagonalizable groups
such that $G$ acts on $X$ via $\phi$. 
Then the map
\[
\bdK _*([X/T]) {\underset {R(T)} \otimes} R(G) \to \bdK _*([X/G])
\]
is an isomorphism.
In particular, $\bdK _0([X/G])$ is a free $R(G)$-module $($and hence a 
free $\Z$-module$)$ if $X$ is $T$-cellular.
\end{cor}
\begin{proof}
To prove the first part of the corollary, 
we trace through the sequence of isomorphisms:
\[
\begin{array}{lll}
\bdK _*([X/T]) {\underset {R(T)} \otimes} R(G) & {\cong} &
\left(\bdK ^T_*(k) {\underset {R(T)} \otimes} \bdK _0([X/T])\right) 
{\underset {R(T)} \otimes} R(G) \\
& {\cong} &
\left(\bdK _*(k) {\underset {\Z} \otimes} R(T)\right)
{\underset {R(T)} \otimes} \left( \bdK _0([X/T]) 
{\underset {R(T)} \otimes} R(G)\right) \\ 
& {\cong}^{\dag} & 
\bdK _*(k) {\underset {\Z} \otimes} \bdK _0([X/G]) \\
& {\cong} & 
\bdK _*([X/G]).
\end{array}
\]

The first and the last isomorphisms in this sequence follow from 
Theorem~\ref{thm:main-thm-2} and the isomorphism ${\cong}^{\dag}$ follows from
Theorem~\ref{thm:main-thm-1}. This proves the first part of the corollary.
If $X$ is $T$-cellular, the freeness of $\bdK _0([X/G])$ as an $R(G)$-module follows 
from Lemma~\ref{lem:linear-P}. 
\end{proof}

\begin{remk}\label{remk:freeness}
In the special case when $[X/G]$ is a smooth toric Deligne-Mumford stack 
(with $X$ projective), the freeness
of $\bdK _0([X/G])$ as $\Z$-module was earlier shown in \cite[Theorem~2.2]{Hua}
and independently in \cite{GHHKK} using symplectic methods.
It is known ({\sl cf.} \cite[Example~4.1]{Hua}) that the freeness
property may fail if $X$ is not projective. 
\end{remk}

\subsection{K-theory of weighted projective spaces}
\label{subsection:WPS}
In the past, there have been many attempts to study the K-theory and
Chow rings of weighted projective spaces. However, there are only a few
explicit computations in this regard. We end this section with an explicit
description of the integral higher K-theory of stacky
weighted projective spaces. 
These are examples of toric stacks, where the spectral sequence
~\eqref{eqn:gen.weak.eq1} degenerates even though $\bdK_0([X/T])$ is not
a projective $R(T)$-module.
We also describe the rational higher G-theory
of weighted projective schemes as another application of 
Theorem~\ref{thm:main-thm-2}.

\subsubsection{Weighted projective spaces}
Let $\underline{q} = \{q_0, \cdots , q_n\}$ be an ordered set of positive 
integers and let $d = gcd(q_0, \cdots , q_n)$.
This ordered set of positive integers gives rise to a morphism of tori
$\phi: \G_m \to (\G_m)^{n+1}$ given by $\phi(\lambda) = (\lambda^{q_0},
\cdots, \lambda^{q_n})$.

The (stacky) weighted projective space $\P(q_0, \cdots ,q_n)$ is the stack  
$[{(\A^{n+1}_k \setminus \{0\})}/{\G_m}]$, where $\G_m$ acts on 
$\A^{n+1}_k$ by $\lambda \cdot (a_0, \cdots , a_n) =
(\lambda^{q_0}a_0, \cdots , \lambda^{q_n} a_n)$. Notice that $\A^{n+1} \setminus
\{0\}$ is a toric variety with dense torus $T = (\G_m)^{n+1}$ acting by
the coordinate-wise multiplication. We see that $\P(\underline{q})$ is the
toric stack associated to the data $((\A^{n+1}_k \setminus \{0\}), 
\G_m \xrightarrow{\phi} T)$. It is known that $\P(\underline{q})$ is a 
Deligne-Mumford toric stack and is reduced (an orbifold) if and only if
$d = 1$.

\subsubsection{{\rm K}-theory of $\P(\underline{q})$}
To describe the higher K-theory of $\P(\underline{q})$, we consider
$\A^{n+1}$ as the toric variety with  dense torus $T = (\G_m)^{n+1}$ acting by
the coordinate-wise multiplication. 
Let $V$ be the $(n+1)$-dimensional representation of $T$ which
represents $\A^{n+1}$ as the toric variety. 
Let $\iota: \Spec(k) \to \A^{n+1}$
and $j: U \to \A^{n+1}$ be the $T$-invariant closed and open inclusions,
where we set $U = \A^{n+1} \setminus \{0\}$. 
Observe that $V$ is the $T$-equivariant normal bundle of 
$\Spec(k)$ sitting inside $\A^{n+1}$ as the origin.

We have the localization exact sequence:
\begin{equation}\label{eqn:WPS0}
\cdots \to \bdK_i([{\Spec(k)}/{\G_m}]) \xrightarrow{\iota_*} 
\bdK_i([{\A^{n+1}}/{\G_m}]) \xrightarrow{j^*} 
\bdK_i([{U}/{\G_m}]) \to \cdots .
\end{equation}

Our first claim is that this sequence splits into short exact sequences
\begin{equation}\label{eqn:WPS1}
0 \to \bdK_i([{\Spec(k)}/{\G_m}]) \xrightarrow{\iota_*} 
\bdK_i([{\A^{n+1}}/{\G_m}]) \xrightarrow{j^*} 
\bdK_i([{U}/{\G_m}]) \to 0
\end{equation}
for each $i \ge 0$.

Using \cite[Proposition~4.3]{VV}, it suffices to show that
$\lambda_{-1}(V) = \stackrel{n}{\underset{i=0}\sum} (-1)^i[\wedge^i(V)]$
is not a zero-divisor in the ring $\bdK_*([{\Spec(k)}/{\G_m}])$.
However, we can write $V = \stackrel{n}{\underset{i = 0}\oplus} V_i$,
where $\G_m$ acts on $V_i \cong k$ by $\lambda \cdot v = \lambda^{q_i}v$.
Since each $q_i$ is positive, we see that no irreducible factor of $V$ is
trivial. It follows from \cite[Lemma~4.2]{VV} that $\lambda_{-1}(V)$ is not
a zero-divisor in the ring $\bdK_*([{\Spec(k)}/{\G_m}])$, and hence
~\eqref{eqn:WPS1} is exact. We have thus proven our claim.

We can now use ~\eqref{eqn:WPS1} to compute $\bdK_*([U/{\G_m}])$.
We first observe that the map $\bdK_*([{\Spec(k)}/{\G_m}]) 
\to \bdK_*([{\A^{n+1}}/{\G_m}])$ induced by the structure map is an 
isomorphism by the homotopy invariance. So we can identify 
the middle term of ~\eqref{eqn:WPS1} with $\bdK_i([{\Spec(k)}/{\G_m}])$.
Furthermore, it follows from the Self-intersection formula
(\cite[Theorem~2.1]{VV}) that the map $\iota_*$ is
multiplication by $\lambda_{-1}(V)$ under this identification.

Since $V = \stackrel{n}{\underset{i = 0}\oplus} V_i$, we get
$\lambda_{-1}(V) = \stackrel{n}{\underset{i = 0}\prod} \lambda_{-1}(V_i)$.
Furthermore, since the class of $V_i$ in $R(\G_m) = \Z[t^{\pm 1}]$ is $t^{q_i}$,
we see that $\lambda_{-1}(V_i) = 1 - t^{q_i}$. We conclude that
$\lambda_{-1}(V)  = \stackrel{n}{\underset{i = 0}\prod} (1 - t^{q_i})$.
We have thus proven: 

\begin{thm}\label{thm:WPS-main} There is a ring isomorphism
\[
\frac{\bdK_*(k)[t^{\pm 1}]}{\stackrel{n}{\underset{i = 0}\prod} (1 - t^{q_i})}
\xrightarrow{\cong} \bdK_*(\P(\underline{q})).
\]
\end{thm}

\begin{remk}\label{remk:Non-proj}
In the above calculations, we can replace $\G_m$ by the dense torus $T$
to get a similar formula. In this case, the exact sequence
~\eqref{eqn:WPS1} shows that $\bdK_0([(\A^{n+1} \setminus \{0\})/T])$
is a quotient of $R(T)$ and hence is not a projective $R(T)$-module.
\end{remk}

\subsubsection{{\rm G}-theory of weighted projective scheme}
The weighted projective scheme is the scheme theoretic
quotient of $\A^{n+1} \setminus \{0\}$ by the above action of $\G_m$.
This is the coarse moduli scheme of $\P(\underline{q})$.
We shall denote this scheme by $\wt{\P(\underline{q})}$.
It is known that this is a normal (but singular in general)
projective scheme. There was no computation available for the higher 
G-theory or K-theory of this schematic weighted projective space.
As an application of Theorem~\ref{thm:main-thm-2}, we now give a
simple description of the rational higher G-theory of $\wt{\P(\underline{q})}$.
We still do not know how to compute its K-theory.

In order to describe the higher G-theory of $\wt{\P(\underline{q})}$,
we shall use the following presentation of this scheme which allows us
to use our main results. We assume that the characteristic of $k$
does not divide any $q_i$.

The torus $T = {\mathbb G} ^n_m$ acts on $\P^n_k$
as the dense open torus by 
$(\lambda_1, \cdots , \lambda_n) \star [z_0, \cdots , z_n]
= [z_0, \lambda_1 z_1, \cdots , \lambda_n z_n]$.
Let $G = \mu_{q_0} \times \cdots \times \mu_{q_n}$ be the product of
finite cyclic groups. Then $G$ acts on $\P^n_k$ by 
$(a_0, \cdots , a_n) \bullet [z_0, \cdots , z_n]
= [a_0z_0, \cdots , a_nz_n]$. It is then easy to see that
$\wt{\P(\underline{q})}$ is isomorphic to the scheme ${\P^n_k}/G$.

Define $\phi : G \to T$ by $\phi(a_0, \cdots , a_n) = 
({a_1}/{a_0}, \cdots , {a_n}/{a_0})$.  
Then one checks that 
\[
\begin{array}{lll}
H  := {\rm Ker}(\phi) & = & \{(a_0, \cdots , a_n) \in G | a_0 = \cdots = a_n\} \\
& = & \{\lambda  \in \G_m| \lambda^{q_0} = 1 = \cdots = \lambda^{q_n}\} \\
& = & \{\lambda \in \G_m| \lambda^d = 1\} \\
& \cong & \mu_d. 
\end{array}
\]

Moreover, it is easy to see that 
\[
\begin{array}{lll}
(a_0, \cdots , a_n) \bullet [z_0, \cdots , z_n] & = &
[a_0z_0, \cdots , a_nz_n] \\
& = & [a^{-1}_0(a_0z_0), \cdots , a^{-1}_0(a_nz_n)] \\
& = & [z_0, ({a_1}/{a_0})z_1, \cdots , ({a_n}/{a_0})z_n] \\
& = & \phi(a_0, \cdots , a_n) \star  [z_0, \cdots , z_n].
\end{array}
\]
In particular, $G$ acts on $\P^n_k$ through $\phi$.
We conclude that $\mathfrak{X} = [\P^{n}_k/G]$ is a smooth toric Deligne-Mumford 
stack associated
to the data $(\P^n_k, G \xrightarrow{\phi} T)$ and there
is an isomorphism $\mathfrak{X} \cong [{\P^n_k}/F] \times 
{\mathfrak{B}}_{\mu_d}$, where $F = {\rm Im}(\phi)$.

\begin{thm}\label{thm:WPS-K-th}
There is a ring isomorphism 
\begin{equation}\label{eqn:WPS0}
\bdK _*(k) {\underset {\Z} \otimes} \frac{[t, t_0, \cdots , t_n]}
{((t-1)^{n+1}, t^{q_0}_0-1, \cdots , t^{q_n}_n-1)} 
\xrightarrow{\cong} \bdK _*(\mathfrak{X}).
\end{equation}
\end{thm}
\begin{proof}
It follows from Corollary~\ref{cor:Base-change} and 
Theorem~\ref{thm:main-thm-2} that there is a ring isomorphism
\[
\bdK _*(k) {\underset {\Z} \otimes} \bdK _0([{\P^n_k}/T]) 
{\underset {R(T)} \otimes} R(G) \xrightarrow{\cong} 
\bdK _*(\mathfrak{X})).
\]
On the other hand, the projective bundle formula implies that the 
left side of this isomorphism is same as
$\bdK _*(k) {\underset {\Z} \otimes} \frac{R(T)[t]}{((t-1)^{n+1})}
{\underset {R(T)} \otimes} R(G)$  which in turn is isomorphic to
$\bdK _*(k) {\underset {\Z} \otimes} \frac{R(G)[t]}{((t-1)^{n+1})}$.
The theorem now follows from the isomorphism
$R(G) \cong \frac{\Z[t_0, \cdots , t_n]}{(t^{q_0}_0-1, \cdots , t^{q_n}_n-1)}$.  
\end{proof}

\begin{cor} There is an isomorphism
\[
\frac{\bdG_*(k)[t]}{((t-1)^{n+1})} \xrightarrow{\cong} 
\bdG_*\left(\wt{\P(\underline{q})}\right)
\]
with the rational coefficients.
\end{cor} 
\begin{proof}
All the groups in this proof will be considered with rational coefficients.
Let $\pi: \P^{n+1}_k \to \wt{\P(\underline{q})}$ be the quotient
map. The assignment $\sF \mapsto \left(\pi_*(\sF)\right)^G$ defines
a covariant functor from the category of $G$-equivariant coherent sheaves
on $\P^{n+1}_k$ to the category of ordinary coherent sheaves on
$\wt{\P(\underline{q})}$. Since the characteristic of $k$ does not divide
the order of $G$, this functor is exact and gives a push-forward map
$\pi_*: {\bdG}^G_*(\P^{n+1}_k) \to \bdG_*\left(\wt{\P(\underline{q})}\right)$. 

Let $\CH^G_*(\P^{n+1}_k)$ denote the equivariant higher Chow groups
of $\P^{n+1}_k$ (\cite{EG1}).
By \cite[Theorem~3]{EG1}, there is a push-forward map
$\ov{\pi}_*: \CH^G_*(\P^{n+1}_k) \to \CH_*\left(\wt{\P(\underline{q})}\right)$
which is an isomorphism.
It follows from \cite[Theorem~9.8, Lemma~9.1]{Krishna0} (see
also \cite[Theorem~3.1]{EG2}) that
there is a commutative diagram

\[
\xymatrix@C3pc{
{\bdG}^G_*(\P^{n+1}_k) {\underset{R(G)}\otimes} \Q \ar[r]^>>>>>>{\tau_G}
\ar[d]_{\pi_*} & \CH^G_*(\P^{n+1}_k) \ar[d]^{\ov{\pi}_*}_{\cong} \\
\bdG_*\left(\wt{\P(\underline{q})}\right) \ar[r]_{\tau}  &
\CH_*\left(\wt{\P(\underline{q})}\right),}
\]
where the horizontal arrows are the Riemann-Roch maps
which are isomorphisms 
(\cite[Theorem~8.6]{Krishna0}).  
It follows that the left vertical arrow is an isomorphism.
The corollary now follows by combining this isomorphism with
Theorem~\ref{thm:WPS-K-th}.
\end{proof}

\section{Toric stack bundles and the stacky Leray-Hirsch theorem}
\label{section:TS-bundle}
Toric bundle schemes and their cohomology were first studied by
Sankaran and Uma in \cite{SU}. They computed the Grothendieck group
of a toric bundle over a smooth base scheme. 
Jiang \cite{Jiang} studied smooth and simplicial Deligne-Mumford toric stack 
bundles over schemes and computed their Chow rings.  
These bundles are relative analogues of toric Deligne-Mumford stacks.
A description of the Grothendieck group of toric Deligne-Mumford stack bundles
was given by Jiang and Tseng in \cite{JTseng2}. 

In this section, we give a general definition of toric stack bundles over
a base scheme in such a way that every fiber of this bundle is a (generically 
stacky) toric stack in the sense of \cite{GSI}.
We prove a stacky version of the Leray-Hirsch theorem for the 
algebraic K-theory of stack bundles. This Leray-Hirsch theorem will
be used in the next section to describe the higher K-theory of toric stack
bundles.

\subsection{Toric stack bundles}\label{subsection:TS-bun-def}
Let $T$ be a split torus of rank $n$ and let $X$ be a scheme with a
$T$-action. Let $G$ be a diagonalizable group over $k$ and let
$\phi: G \to T$ be a morphism of algebraic groups over $k$.

Let $p: E \to B$ be a principal $T$-bundle over a scheme $B$. 
Let $G$ act on $E \times X$ 
by $g(e,x) = (e,gx):= (e, \phi(g)x)$ and let $T$ act on $E \times X$ via the
diagonal action. It is easy to see that these two actions commute
and the projection map $E \times X \to E$ is equivariant with respect to 
these actions. 

The commutativity of the actions ensures that the
$G$-action descends to the quotients $E(X) : = E \stackrel{T}{\times} X$ and 
$E/T = B$ such that the induced map of quotients 
$\ov{p} : E(X) \to B$ is $G$-equivariant.
Since $E$ has trivial $G$-action, so does $B$ and  we see that $G$ acts on 
$E(X)$ fiber-wise
and the map $\ov{p}$ canonically factors through the stack quotient
$\pi: [E(X)/G] \to B$.
Notice that $E$ is a Zariski locally trivial $T$-bundle and so are
$E(X) \to B$ and $[E(X)/G] \to B$. Setting $\mathfrak{X} = [E(X)/G]$,
we conclude that the map $\pi: \mathfrak{X} \to B$ is a Zariski locally
trivial fibration each of whose fiber is the stack $[X/G]$. The morphism
$\pi$ will be called a {\sl stack bundle} over $B$. 

If $X$ is a toric variety with dense torus $T$, then $\pi: \mathfrak{X} \to B$
will be called a {\sl toric stack bundle} over $B$. In this case,
each fiber of $\pi$ is the toric stack $[X/G]$ in the sense of \cite{GSI}.
If $[X/G]$ is a Deligne-Mumford stack, this construction recovers the notion of
toric stack bundles used in \cite{Jiang} and \cite{JTseng2}.

\subsection{Leray-Hirsch Theorem for stack bundles}
\label{subsection:LRT}
First we prove the following lemma.
\begin{lem}\label{lem:linear-P}
Let $X$ be a $T$-equivariantly cellular scheme with the $T$-equivariant
cellular decomposition
\begin{equation}\label{eqn:linear-P*}
\emptyset = X_{n+1} \subsetneq X_n \subsetneq \cdots \subsetneq X_1
\subsetneq X_0 = X
\end{equation}
and let $U_i = X\setminus X_i$ for $0 \le i \le n+1$.
Let $G$ be a diagonalizable group provided with a morphism of algebraic groups
$\phi: G \to T$.
Then for any $0 \le i \le n$, the sequence 
\begin{equation}\label{eqn:linear-P0}
0 \to \bdG ^G_*\left(U_{i+1} \setminus U_i\right) \to \bdG ^G_*(U_{i+1}) \to
\bdG ^G_*(U_i) \to 0
\end{equation}
is exact. In particular, $\bdG ^G_0(X)$ is a free $R(G)$-module of rank
equal to the number of $T$-invariant affine cells in $X$ with basis given by 
the closures of the affine cells.
\end{lem}
\begin{proof}
To prove the exactness part of the proposition, 
we first make the following claim. 
Suppose $X$ is
a $G$-scheme and $j:U \inj X$ is a $G$-invariant open inclusion with
complement $Y$. Suppose that $U$ is isomorphic to a representation of $G$.
Then the localization sequence 
\begin{equation}\label{eqn:linear-P1}
0 \to \bdG ^G_*(Y) \to \bdG ^G_*(X) \xrightarrow{j^*} \bdG ^G_*(U) \to 0
\end{equation}
is (split) short exact.

To prove the claim, let $\alpha: X \to \Spec(k)$ and $\beta: U \to \Spec(k)$
be the structure maps (which are $G$-equivariant) so that
$\beta = \alpha \circ j$. The homotopy invariance of equivariant K-theory
shows that $\beta^*$ is an isomorphism. Let $\gamma = \alpha^* \circ 
(\beta^*)^{-1}$. Then one checks that $\gamma$ is a section of $j^*$ and
hence the localization sequence splits into short exact sequences.
This proves the claim.

We shall prove ~\eqref{eqn:linear-P0} by induction on the number of 
$T$-invariant affine cells in $X$. 
For $i = 0$, ~\eqref{eqn:linear-P0} is immediate.
So we assume $i \ge 1$ and consider the commutative diagram:
\begin{equation}\label{eqn:linear-P2}
\xymatrix@C1pc{
& 0 \ar[d] & 0 \ar[d] & 0 \ar[d] & \\
0 \ar[r] & \bdG ^G_*(X_i \setminus X_{i+1}) \ar[r] \ar@{=}[d] &
\bdG ^G_*(X_1 \setminus X_{i+1}) \ar[r] \ar[d] & 
\bdG ^G_*(X_1 \setminus X_{i}) \ar[r] \ar[d] & 0 \\
0 \ar[r] & \bdG ^G_*(X_i \setminus X_{i+1}) \ar[r]  &
\bdG ^G_*(X \setminus X_{i+1}) \ar[r] \ar[d] & 
\bdG ^G_*(X \setminus X_{i}) \ar[r] \ar[d] & 0 \\ 
& &
\bdG ^G_*(X \setminus X_1) \ar@{=}[r] \ar[d] & 
\bdG ^G_*(X \setminus X_1) \ar[r] \ar[d] & 0 \\
& & 0  & 0.}
\end{equation}

The top row is exact by induction on the number of affine cells since
$X_1$ is $T$-equivariantly cellular with fewer number of cells.
The two columns are exact by the above claim. It follows that the middle
row is exact, which proves ~\eqref{eqn:linear-P0}.

To prove the last (freeness) assertion, we apply ~\eqref{eqn:linear-P1}
to the inclusion $X_1 \subset X$ and see that $\bdG ^G_0(X) \cong
\bdG ^G_0(X_1) \oplus R(G)$. An induction on the number of affine $G$-cells
now finishes the proof.
\end{proof}

\begin{prop}\label{prop:linear}
Let $X$ be a $T$-equivariantly cellular scheme and let $B$ be any
scheme with trivial $T$-action. Then the external product map
\begin{equation}\label{eqn:linear1}
\bdG _*(B) {\otimes}_{\Z} \bdG ^G_0(X)     
\to \bdG ^G_*(B \times X)
\end{equation}
is an isomorphism. In particular, the natural map
$\bdK _*(k) {\otimes}_{\Z} \bdG ^G_0(X) \to \bdG ^G_*(X)$ is an
isomorphism.
\end{prop}
\begin{proof} 
Since the map
\begin{equation}\label{eqn:linear1*}
\bdG_ *(B) \otimes_{\Z} R(G) \xrightarrow{\cong} \bdG ^G_*(B)
\end{equation}
is an isomorphism ({\sl cf.} \cite[Lemma~5.6]{Thomason1}), 
the lemma is equivalent to the assertion that the map
\begin{equation}\label{eqn:linear1*0}
\bdG ^G_*(B) {\otimes}_{R(G)} \bdG ^G_0(X) \to \bdG ^G_*(B \times X)
\end{equation}
is an isomorphism.

Consider the cellular decomposition of $X$ as in 
Lemma~\ref{lem:linear-P}.
Then each $U_i = X\setminus X_i$ is also a $T$-equivariantly cellular scheme. 
It suffices to show by induction on $i \ge 0$ 
that ~\eqref{eqn:linear1*0} holds when $X$ is any of these $U_i$'s. 
There is nothing to prove for $i =0$ and the case $i = 1$
follows by the homotopy invariance since $U_1$ is an affine space.

To prove the  general case, we use the short exact sequence
\begin{equation}\label{eqn:linear2}
0 \to \bdG ^G_0\left(U_{i+1} \setminus U_i\right) \to \bdG ^G_0(U_{i+1}) \to
\bdG ^G_0(U_i) \to 0
\end{equation}
given by Lemma~\ref{lem:linear-P}. This sequence
splits, since each 
$\bdG ^G_0(U_i)$ was shown to be free over $R(G)$ in Lemma~\ref{lem:linear-P}. 
Tensoring this with $\bdG ^G_*(B)$ over $R(G)$, we obtain a commutative diagram 
\[
\xymatrix@C.3pc{
0 \ar[r] & \bdG ^G_*(B) {\otimes} \bdG ^G_0\left(U_{i+1} \setminus U_i\right)  
\ar[r] \ar[d] & \bdG ^G_*(B) {\otimes} \bdG ^G_0(U_{i+1}) 
\ar[r] \ar[d] & \bdG ^G_*(B) {\otimes} \bdG ^G_0(U_i) 
\ar[r] \ar[d] & 0 \\
\ar[r] &  \bdG ^G_*(B \times (U_{i+1} \setminus U_i)) \ar[r]_{\ \ \ i_*} &
\bdG ^G_*(B \times U_{i+1}) \ar[r]_{j^*} &  
\bdG ^G_*(B \times U_i) \ar[r] &}
\] 
where the top row remains exact since the short exact sequence in 
~\eqref{eqn:linear2}
is split. The bottom row is the localization exact sequence.
The left vertical arrow is an isomorphism by the homotopy invariance
and the right vertical arrow is an isomorphism by the induction.
In particular, $j^*$ is surjective in all indices. We conclude that
$i_*$ is injective in all indices and the middle vertical arrow is  
an isomorphism. 
\end{proof}

\begin{thm}$($Stacky Leray-Hirsch theorem$)$\label{thm:LHT}
Suppose that $k$ is a perfect field and $B$ is a smooth scheme over $k$.
Let $X$ be a $T$-equivariantly cellular scheme. 
Let $\mathfrak{F} \xrightarrow{i} \mathfrak{X} \xrightarrow{\pi} B$ 
be a Zariski locally trivial stack bundle (\S~\ref{subsection:TS-bun-def})
each of whose fiber $\mathfrak{F}$ is a smooth stack of the form $[X/G]$. 
Assume that 
there are elements $\{e_1, \cdots , e_r\}$ in $\bdK _0(\mathfrak{X})$ such that 
$\{f_1 = i^*(e_1), \cdots , f_r = i^*(e_r)\}$ is an $R(G)$-basis of
$\bdK _0(\mathfrak{X}_b)$ for each fiber $\mathfrak{X}_b = \mathfrak{F}$ of the 
fibration. 
Then the map
\begin{equation}\label{eqn:LHT**}
\Phi : \bdK _0(\mathfrak{F}) {\underset{R(G)}\otimes} \bdK ^G_*(B) \to 
\bdK _*(\mathfrak{X})
\end{equation}
\[
\Phi\left({\underset{1 \le i \le r}\sum} \ f_i \otimes b_i\right)
=  {\underset{1 \le i \le r}\sum} {\pi}^*(b_i) e_i
\]
is an isomorphism of $R(G)$-modules. 
In particular, $\bdK _*(\mathfrak{X})$ is a free $\bdK ^G_*(B)$-module and
the map $\pi^*: \bdK ^G_*(B) \to \bdK _*(\mathfrak{X})$ is injective.
\end{thm}
\begin{proof}
Since $k$ is perfect and since the fibration $p$ is Zariski locally trivial,
we can find a filtration 
\begin{equation}\label{eqn:LHT-fil}
\emptyset = B_{n+1} \subsetneq B_n \subsetneq \cdots \subsetneq B_1
\subsetneq B_0 = B
\end{equation}
of $B$ by closed subschemes such that for each $0 \le i \le n$,
the scheme $B_i \setminus B_{i+1}$ is smooth and the given fibration is trivial 
over it. We set $U_i = B \setminus B_i$ 
and $V_i = U_i \setminus U_{i-1} = B_{i-1} \setminus B_i$. 
Observe then that each of $U_i$'s and $V_i$'s is smooth. 

Set $\mathfrak{X}_i = {\pi}^{-1}(U_i)$ and $\mathfrak{W}_i = {\pi}^{-1}(V_i) = 
V_i \times \mathfrak{F}$. 
Let $\eta_i : \mathfrak{X}_i \inj \mathfrak{X}$ and 
$\iota_i: \mathfrak{W}_i \inj\mathfrak{X}$ be the inclusion maps. 
We prove by induction on $i$ that the map
$\bdK _0(\mathfrak{F}) {\underset{R(G)}\otimes} \bdK ^G_*(U_i) \to 
\bdK _*(\mathfrak{X}_i)$ 
is an isomorphism, which will
prove the theorem. 

Since $U_0 = \emptyset$ and 
$\mathfrak{X}_1 = U_1 \times \mathfrak{F}$, the desired isomorphism for 
$i \le 1$ follows from Proposition~\ref{prop:linear} and the isomorphism 
$U_1 \times \mathfrak{F} \cong [(U_1 \times X)/G]$. 
We now consider the commutative diagram: 
\begin{equation}\label{eqn:LHT&}
\xymatrix@C.5pc{
{\begin{array}{c}
\bdK ^G_*(U_{i}) \\
{\otimes} \\
\bdK _0(\mathfrak{F})
\end{array}} 
\ar[r] \ar[d] & 
{\begin{array}{c}
\bdK ^G_*(V_{i+1}) \\
{\otimes} \\
\bdK _0(\mathfrak{F})
\end{array}}
\ar[d] \ar[r] &
{\begin{array}{c}
\bdK ^G_*(U_{i+1}) \\
{\otimes} \\
\bdK _0(\mathfrak{F})
\end{array}} 
\ar[d] \ar[r] &
{\begin{array}{c}
\bdK ^G_*(U_{i}) \\
{\otimes} \\
\bdK _0(\mathfrak{F})
\end{array}} 
\ar[r] \ar[d] & 
{\begin{array}{c}
\bdK ^G_*(V_{i+1}) \\
{\otimes} \\
\bdK _0(\mathfrak{F})
\end{array}}
\ar[d] \\
\bdK _*(\mathfrak{X}_i) \ar[r] & \bdK _*(\mathfrak{W}_{i+1}) \ar[r] & 
\bdK _*(\mathfrak{X}_{i+1}) \ar[r] & 
\bdK _*(\mathfrak{X}_i) \ar[r] & \bdK _*(\mathfrak{W}_{i+1}).}
\end{equation}

The top row in this diagram is obtained by tensoring the K-theory long exact 
localization sequence with $\bdK _0(\mathfrak{F})$ over $R(G)$, and the 
bottom row
is just the localization exact sequence. Since $\bdK _0(\mathfrak{F})$ is a free
$R(G)$-module ({\sl cf.} Lemma~\ref{lem:linear-P}), the top row is also exact. 

It is easily checked that the second and the third squares commute
using the commutativity property of the push-forward and pull-back maps
of K-theory of coherent sheaves in a Cartesian diagram of proper and
flat maps. 
We show that the other squares also commute. It is enough to show that
the first square commutes as the fourth one is same as the first.
Let $\delta$ denote the connecting homomorphism in a long exact localization
sequence for higher K-theory. 

If we start with an element $b \otimes i^*(e_j) \in \bdK _*(U_{i}) {\otimes} 
\bdK _0(\mathfrak{F})$ and map this horizontally, we obtain 
$\delta b \otimes i^*(e_j)$
which maps vertically down to ${\pi}^*(\delta b) \cdot \iota^*_{i+1}(e_j)$. 
On the other hand,
if we first map vertically, we obtain ${\pi}^*(b) \cdot \eta^*_i(e_j)$ which 
maps horizontally to $\delta \left({\pi}^*(b) \cdot \eta^*_i(e_j) \right)$. 

Now, we recall that these elements in the higher K-theory of coherent
sheaves are represented by the elements in the higher homotopy groups of the
various infinite loop spaces. Moreover, if we have a closed immersion
of smooth stacks $\mathfrak{F} \inj \mathfrak{X}$ with open complement
$\mathfrak{U}$, then we have a fibration sequence of ring spectra
\begin{equation}\label{eqn:Ring-sp}
\bdK (\mathfrak{F}) \to  \bdK (\mathfrak{X}) \to \bdK (\mathfrak{U}).
\end{equation}

The homotopy groups of these ring spectra form graded rings and the
connecting homomorphism in the long exact sequence of the homotopy groups 
associated to the above
fibration sequence satisfies the Leibniz rule 
(e.g., see \cite[Appendix~A]{Brown} and \cite[\S~2.4]{Panin}).

Applying this Leibniz rule, we see that the term
$\delta \left({\pi}^*(b) \cdot \eta^*_i(e_j) \right)$ is same as 
$\delta {\pi}^*(b) \cdot \iota^*_{i+1}(e_j) = {\pi}^*\left(\delta b\right) \cdot 
\iota^*_{i+1}(e_j)$ since 
$\delta (\eta^*_i(e_j)) = 0$. We have shown that the above diagram commutes.

The first and the fourth vertical arrows in ~\eqref{eqn:LHT&} are isomorphisms 
by induction. The second and the fifth vertical arrows are isomorphisms by 
Proposition~\ref{prop:linear}. 
Hence the middle vertical arrow is also an isomorphism by 5-lemma.

To show that $\pi^*$ is injective, consider the $T$-invariant filtration of $X$
as in ~\eqref{eqn:linear-P*} and let $j: [E(U_1)/G] = \mathfrak{X}_1
\to \mathfrak{X}$ be the open inclusion. 
If we apply ~\eqref{eqn:LHT**} to the map $ \mathfrak{X}_1 \to B$, we see
that the composite map $\bdK ^G_*(B) \to \bdK _*(\mathfrak{X}) \to
\bdK _*(\mathfrak{X}_1)$ is an isomorphism (since $U_1$ is a $T$-invariant 
cell of $X$). We conclude that $\pi^*$ is split injective.
\end{proof}

\section{Higher K-theory of toric stack bundles}
\label{section:K-Chow-TSB}
In this section, we give explicit descriptions of the higher K-theory of 
toric stack bundles in terms of the higher K-theory of the base scheme. 

Let $T$ be a split torus of rank $n$. Let $N = \Hom(\G_m, T)$ be the lattice of
one-parameter subgroups of $T$ and let $M = \Hom(T, \G_m) = N^{\vee}$ be
its character group. Let $X = X(\Delta)$ be a smooth
projective toric variety associated to a fan $\Delta$ in $N_{\R}$. Let
 
\begin{equation}\label{eqn:reduced}
0 \to G \to T \to T' \to 0
\end{equation}
be an exact sequence of diagonalizable groups.
This yields the exact sequence of the character groups
\begin{equation}\label{eqn:reduced}
0 \to T'^{\vee} \to T^{\vee} \to G ^{\vee} \to 0.
\end{equation}

\subsection{The Stanley-Reisner algebra associated to a 
subgroup of $T$}\label{subsection:SRA}
We fix an ordering $\{\sigma_1, \cdots , \sigma_m\}$ of $\Delta_{\rm max}$  
and let $\tau_i \subset \sigma_i$ be the cone which is the
intersection of $\sigma_i$ with all those $\sigma_j$ such that $j \ge i$
and which intersect $\sigma_i$ in dimension $n-1$. 
Let $\tau'_i \subset \sigma_i$ be the cone such that $\tau_i \cap \tau'_i 
= \{0\}$ and ${\rm dim}(\tau_i) + {\rm dim}(\tau'_i) = n$ for $1 \le i \le m$.
It is easy to see that $\tau'_i$ is the intersection of $\sigma_i$ with all 
those $\sigma_j$ such that $j \le i$ and which intersect $\sigma_i$ in 
dimension $n-1$. 
Since $X$ is smooth and projective, it is well known that we can choose the
above ordering of $\Delta_{\rm max}$ such that
\begin{equation}\label{eqn:order}
\tau_i \subset \sigma_j \Rightarrow \ i \le j \ \ {\rm and}
\ \ \tau'_i \subset \sigma_j  \Rightarrow  \ j \le i.
\end{equation}

Let $\Delta_1 = \{\rho_1, \cdots , \rho_d\}$ be the set of one-dimensional 
cones in $\Delta$ and let $\{v_1, \cdots , v_d\}$ be the associated primitive
elements of $N$. We choose $\{\rho_1, \cdots , \rho_n\}$ to be a set of one 
dimensional faces of $\sigma_m$ such that $\{v_1, \cdots , v_n\}$ is a 
basis of $N$. Let $\{\chi_1, \cdots , \chi_n\}$ be the dual basis of $M$.
Let $\{\chi'_1, \cdots , \chi'_r\}$ be a chosen basis of $T'^{\vee} = M'$.
We will denote the group operations in all the lattices additively.

\begin{defn}\label{defn:RING}
Let $A$ be a commutative ring with unit and let $\{r_1, \cdots ,r_n\}$ be 
a set of invertible elements in $A$. 
Let $I^T_{\Delta}$ denote the ideal of the Laurent polynomial algebra 
$A[t^{\pm 1}_1, \cdots , t^{\pm 1}_d]$ generated by the elements
\begin{equation}\label{eqn:Reln1}
 (t_{j_1}-1) \cdots (t_{j_l}-1), \ 1 \le j_p \le d
\end{equation}
such that $\rho_{j_1}, \cdots , \rho_{j_l}$ do not span a cone of $\Delta$.
Let $J^G_\Delta$ denote the ideal of $A[t^{\pm 1}_1, \cdots , t^{\pm 1}_d]$ 
generated by the relations
\begin{equation}\label{eqn:Reln2}
s_i := \left(\stackrel{d}{\underset{j = 1}\prod}
(t_j)^{<-\chi'_i, v_j>}\right) - r_i, \ 1 \le i \le r.
\end{equation}
We define the $A$-algebras $R_T(A, \Delta)$ and $R_G(A, \Delta)$ to be 
quotients of ${A[t^{\pm 1}_1, \cdots , t^{\pm 1}_d]}$ by the ideals $I^T_\Delta$ and 
$I^G_\Delta = I^T_{\Delta} + J^G_{\Delta}$, respectively. 
\end{defn}

The ring $R_G(A, \Delta)$ will be called
the {\sl Stanley-Reisner} algebra over $A$ associated to the subgroup $G$.
Every character $\chi \in M$ acts on $R_T(A, \Delta)$ via
multiplication by the element $t_\chi =  
\left(\stackrel{d}{\underset{j = 1}\prod} (t_j)^{<-\chi, v_j>}\right)$
and this makes $R_T(A, \Delta)$ (and hence $R_G(A, \Delta)$)
an $\left(A {\underset{\Z}\otimes} R(T)\right)$-algebra. 

\vskip .3cm 

\subsection{The K-theory of toric stack bundles}
\label{subsection:Main-formula}
Let $T$ be a split torus over a perfect field $k$ and let $G$ 
be a closed subgroup of $T$ (which may not necessarily be a torus).
Let $X$ be a smooth projective toric variety with dense torus $T$
and let $\pi:\mathfrak{X} = [(E(X)/G] \to B$ be a toric stack bundle
over a smooth $k$-scheme $B$ associated to a principal $T$-bundle
$p: E \to B$. We wish to describe the K-theory of $\mathfrak{X}$
in terms of the K-theory of $B$.

\vskip .3cm 
 
Any $T$-equivariant
line bundle $L \to X$ uniquely defines a $G$-equivariant line bundle 
$E(L) = E \stackrel{T}{\times} L$ on $E(X)$, where the $G$-action on
$E(L)$ is given exactly as on $E(X)$.  
Every $\rho \in \Delta_1$ 
defines a unique $T$-equivariant line bundle
$L_{\rho}$ on $X$ with a $T$-equivariant section $s_{\rho} : X \to L_{\rho}$
which is transverse to the zero-section and whose zero locus is the orbit
closure $V_\rho = \ov{O_\rho}$. 

For any $\sigma \in \Delta$, let $u_{\sigma}$ 
denote the fundamental class $[\sO_{V_\sigma}]$  of the $T$-invariant subscheme
$V_{\sigma}$ in $\bdK ^T_0(X)$ and let $y_{\sigma}$ denote the fundamental class
of $[E\left(V_\sigma\right)]$ in $\bdK ^G_0(E(X)) = \bdK _0(\mathfrak{X})$.

Notice that ${\ov{p}}_{\sigma}: E\left(V_\sigma\right) \to B$ is a 
$G$-equivariant smooth 
projective toric sub-bundle of $\ov{p}:E(X) \to B$ with fiber 
$V_\sigma$. In particular, $\pi_{\sigma} :[E\left(V_\sigma\right)/G] \to B$
is a toric stack sub-bundle of $\pi: \mathfrak{X} \to B$ with fiber 
$[V_\sigma/G]$. We set $\mathfrak{X}_{\sigma} = [E\left(V_\sigma\right)/G]$.

Suppose that $\rho_{j_1}, \cdots , \rho_{j_l}$ do not span a cone in $\Delta$.
Then $s = (s_{j_1}, \cdots , s_{j_l})$ yields a $G$-equivariant nowhere 
vanishing section of
$E(L_{\rho_{j_1}}) \oplus \cdots \oplus E(L_{\rho_{j_l}})$ and hence the Whitney sum
formula for Chern classes in K-theory implies that

\begin{equation}\label{eqn:vanish1}
y_{\rho_{j_1}} \cdots y_{\rho_{j_l}} = 0 \ {\rm in} \ 
\bdK ^G_0\left(E(X)\right). 
\end{equation}

We now consider the commutative diagram
\begin{equation}\label{eqn:vanish2}
\xymatrix@C.7pc{
X_l \ar[r]^{\iota} \ar[d]_{\pi_l} & E(X) \ar[d]_{\ov{p}} & 
E \times X \ar[d]^{p_E} \ar[r]^>>>>{p_X} \ar[l]_{p'} & 
X \ar[d]^{\pi_X} \\
{\rm Spec}(l) \ar[r] & B & E \ar[l]^{p} \ar[r]_<<<<{\pi_E} & {\rm Spec}(k),}
\end{equation}
where ${\rm Spec}(l)$ is any point of $B$. It is clear that all squares are 
Cartesian and all the maps in the right square are
$T$-equivariant.

We define $(T \times G)$-actions on any $T$-invariant subscheme $Y \subseteq X$
and on $E$ by $(t,g)\cdot y = tg \cdot y$ and
$(t,g)\cdot e = t\cdot e$, respectively. An action of $(T \times G)$ on
$E \times X$ is defined by $(t,g) \cdot (e,x) = (t\cdot e, tg \cdot x)$.
It is clear that these are group actions such that the square on the right in
~\eqref{eqn:vanish2} is $(T\times G)$-equivariant. This implies that
the middle square is also $(T\times G)$-equivariant and the map
$\ov{p}$ is $G$-equivariant with respect to the trivial action of $G$ on $B$.
The square on the left is $G$-equivariant.

Let $L_{\chi}$ denote the $T$-equivariant line bundle on
${\rm Spec}(k)$ associated to a character $\chi$ of $T$. 
Let $(T \times G)$ act on $L_{\chi}$ by $(t,g)\cdot v = \chi(t) \chi(g) \cdot v$.
If $\chi \in M' = T'^{\vee}$, then $G$ acts trivially on $L_{\chi}$ and hence
it acts trivially on $\pi^*_E(L_{\chi})$. Recall that $(T \times G)$ acts on
$E$ via $T$. Hence $\pi^*_E(L_{\chi}) \to E$ is a $(T \times G)$-equivariant
line bundle on which $G$-acts trivially. 
Since the $T$-equivariant line
bundles on $E$ are same as ordinary line bundles on $B$, 
we find that for every $\chi \in M'$, there is a unique ordinary line
bundle $\zeta_{\chi}$ on $B$ such that $\pi^*_E(L_{\chi}) = p^*(\zeta_{\chi})$.

Since $G$ acts trivially on $B$, there is a canonical ring homomorphism
$c_B: \bdK _*(B) \to \bdK ^G_*(B)$ such that the composite $\bdK _*(B) \xrightarrow{c_B}
\bdK ^G_*(B) \to \bdK _*(B)$ is identity. These maps are simply the
maps $\bdK _*(B) \xrightarrow{c_B} \bdK ^G_*(B) = \bdK _*(B)\otimes_{\Z} R(G) \to \bdK _*(B)$.
Since $p^*_X \circ \pi^*_X (L_{\chi}) = p^*_E \circ \pi^*_E (L_{\chi})$
and since the $(T \times G)$-equivariant vector bundles on $E \times X$
are same as $G$-equivariant vector bundles on $E(X)$,
we conclude that for every $\chi \in M'$, there is a unique ordinary line
bundle $\zeta_{\chi}$ on $B$ such that 
\begin{equation}\label{eqn:vanish3}
E(\pi^*_X(L_{\chi})) = {\ov{p}}^*(\zeta_{\chi}) = 
{\ov{p}}^*\left(c_B(\zeta_{\chi})\right). 
\end{equation}

Notice also that on each open subset of $B$ where the bundle $p$ is trivial,
the restriction of $\zeta_{\chi}$ is the trivial line bundle since
$\zeta_{\chi}$ is obtained from the $T$-line bundle $L_{\chi}$ on $\Spec(k)$.

We define a homomorphism of $\bdK _*(B)$-algebras 
$\bdK _*(B)[t^{\pm 1}_1, \cdots , t^{\pm 1}_d] \to
\bdK _*(\mathfrak{X})$ by the assignment $t_i \mapsto [{E(L^{\vee}_{\rho_i})}/G]$
for $1 \le i \le d$. If we let $r_i = \zeta_{\chi'_i}$ for
$1 \le i \le r$ ({\sl cf.} \S~\ref{subsection:SRA}), 
then it follows from ~\eqref{eqn:vanish1} and 
~\eqref{eqn:vanish3} that this homomorphism descends to a $\bdK _*(B)$-algebra
homomorphism
\begin{equation}\label{eqn:vanish4}
\Phi_G : R_G\left(\bdK _*(B), \Delta \right) \to \bdK _*(\mathfrak{X}).
\end{equation}

\vskip .3cm 

Given a sequence $\gamma = \{i_1, \cdots, i_d\}$ of integers,
set $E(\gamma) = E\left((L^{\vee}_{\rho_1})^{i_1} \otimes \cdots \otimes
(L^{\vee}_{\rho_d})^{i_d}\right)$. We then see that
for a monomial $\gamma(\underline{t}) = t^{i_1}_1\cdots t^{i_d}_d$, we have
\begin{equation}\label{eqn:vanish4*}
\Phi_G(\gamma(\underline{t})) = [{E(\gamma)}/G].
\end{equation}

The following result describes the higher K-theory of the toric stack
bundle $\pi: \mathfrak{X} \to B$. 

\begin{thm}\label{thm:CTB}
The homomorphism $\Phi_G$ is an isomorphism. 
\end{thm}

Before we prove this theorem, we consider some special cases
which will be used in the final proof. The following observations
will be used throughout the proofs.

The first observation is that the
cell closures of $X$ are the $T$-equivariant subschemes
$V_{\tau_i}$. So the classes of $\sO_{V_{\tau_i}}$ form an $R(G)$-basis of
$\bdK ^G_0(X)$ by Lemma~\ref{lem:linear-P}.
Since $\iota^*(y_{\tau_i}) = [\sO_{V_{\tau_i}}]$, we see that 
Theorem~\ref{thm:LHT} applies to the toric stack 
bundle $\pi:\mathfrak{X} \to B$.

Second observation is that $G$ is a diagonalizable group which acts trivially
on $B$. Hence the map $\bdK _*(B) {\underset{\Z}\otimes} R(G) \to \bdK ^G_*(B)$
is a ring isomorphism by \cite[Lemma~3.6]{Thomason1}.
This identification will be used without further mention.
Since any character $\chi \in M$ acts on $R_T(\bdK _*(B), \Delta)$ and 
$\bdK ^G_*(E(X))$ via multiplication by $t_{\chi}$ and $\Phi_G(t_{\chi})$
respectively ({\sl cf.} \cite[Proposition~4.3]{SU}),
we observe that the composite map $R_T(\bdK _*(B), \Delta) \to 
R_G(\bdK _*(B), \Delta) \to \bdK ^G_*(E(X))$ is $\bdK ^T_*(B)$-linear.

\begin{remk}\label{remk:Thomason}
We remark that the result of Thomason in \cite[Lemma~3.6]{Thomason1}
is stated for affine schemes, but his proof works for all schemes.
Another way to deduce the general case from the affine case is to
get a stratification of $B$ by affine subschemes as in ~\eqref{eqn:LHT-fil},
use induction on the number of affine strata, the localization sequence
and the fact that $R(G)$ is free over $\Z$.   
\end{remk}

\vskip .3cm

\begin{lem}\label{lem:T-case}
The homomorphism $\Phi_G$ is an isomorphism when $G = T$. 
\end{lem}
\begin{proof}
In this case, we first notice that the map $R_T(\Z, \Delta) 
\xrightarrow{\phi} \bdK ^T_0(X)$ which takes $t_i$ to $[L^{\vee}_{\rho_i}]$, 
is an isomorphism of $R(T)$-algebras by \cite[Theorem~6.4]{VV}. 
On the other hand, we have the maps 
\begin{equation}\label{eqn:CTB0*}
\bdK _*(B) {\underset{\Z}\otimes} R_T(\Z, \Delta) \xrightarrow{\cong}
R_T(\bdK _*(B), \Delta)  \xrightarrow{\Phi_T}  \bdK ^T_*(E(X)),
\end{equation}
where the first map takes $\alpha \otimes t_i$ to $\alpha \cdot t_i$
for $1 \le i \le d$. This map is clearly an isomorphism 
(see ~\eqref{eqn:Reln1}).
It is clear from the definition of $\Phi_T$ that the composite map
is same as the map $\Phi$ in ~\eqref{eqn:LHT**} (with $G = T$).
It follows from Theorem~\ref{thm:LHT} that the composite map
in ~\eqref{eqn:CTB0*} is an isomorphism. We conclude that $\Phi_T$ is an
isomorphism.
\end{proof}

\begin{cor}\label{cor:T-case*}
For any closed subgroup $G \subseteq T$,
the ring $R_G(\bdK _*(B), \Delta)$ is a free $\bdK _*(B)$-module.
\end{cor}
\begin{proof}
We have seen above that the image of a
character $\chi \in M$ in $R_T(\bdK _*(B), \Delta)$ is $t_{\chi}$.
If we let $J^G$ denote the ideal $\left(\chi'_1 - \zeta_{\chi'_1}, \cdots ,
\chi'_r - \zeta_{\chi'_r}\right)$ in $\bdK ^T_*(B)$, then it follows from 
~\eqref{eqn:Reln2} that $J^G_{\Delta} = J^GR_T(\bdK _*(B), \Delta)$ under the map
$\bdK ^T_*(B) \to R_T(\bdK _*(B), \Delta)$.

It follows from Lemma~\ref{lem:T-case} and Theorem~\ref{thm:LHT}
(with $G=T$) that $R_T(\bdK _*(B), \Delta)$ is a free $\bdK ^T_*(B)$-module. 
This implies that $R_G(\bdK _*(B), \Delta) = 
{R_T(\bdK _*(B), \Delta)}/{J^G_{\Delta}}$
is a free ${\bdK ^T_*(B)}/{J^G}$-module.
Thus, it suffices to show that ${\bdK ^T_*(B)}/{J^G}$ is a free 
$\bdK _*(B)$-module.
Since $\bdK ^T_*(B)$ is isomorphic to a Laurent polynomial ring 
$\bdK _*(B)[x^{\pm 1}_1, \cdots, x^{\pm 1}_n]$ and since each character 
$\chi \in M'$ 
is a monomial in this ring, the desired freeness follows from 
Lemma~\ref{lem:easy}.
\end{proof}

\begin{lem}\label{lem:Trvial-bundle-case}
The homomorphism $\Phi_G$ is an isomorphism when $p: E \to B$ is a trivial
principal bundle.
\end{lem}
\begin{proof}
Since $p: E \to B$ is a trivial bundle, we have observed before that
$\zeta_{\chi'_i} =1$ for each $1 \le i \le r$. In particular,
the map ${\bdK ^T_*(B)}/{J^G} \to \bdK ^G_*(B)$ is an isomorphism 
by Lemma~\ref{lem:Groupring}, where $J^G$ is as in Corollary~\ref{cor:T-case*}.

It follows from Theorem~\ref{thm:LHT} and Lemma~\ref{lem:T-case} that
$\Phi_T$ is an isomorphism of free $\bdK ^T_*(B)$-modules.
This implies that $R_G(\bdK _*(B), \Delta) = 
{R_T(\bdK _*(B), \Delta)}/{J^G_{\Delta}}$ 
is a free ${\bdK ^T_*(B)}/{J^G} = \bdK ^G_*(B)$-module.
It follows from this and Theorem~\ref{thm:LHT} that 
$\Phi_G$ is a basis preserving homomorphism of
free $\bdK ^G_*(B)$-modules of same rank. Hence, it must be an isomorphism.
\end{proof}

\begin{lem}\label{lem:easy}
Let $S = A[x^{\pm 1}_1, \cdots, x^{\pm 1}_n]$ be a Laurent polynomial ring over
a commutative ring $A$ with unit. Let $\{t_1, \cdots , t_r\}$
be a set of monomials in $S$ and let $\{u_1, \cdots , u_r\}$ be a set of
units in $A$. Then the ring $\frac{S}{(t_1 - u_1, \cdots , t_r-u_r)}$
is free over $A$.
\end{lem}
\begin{proof}
This is left as an easy exercise using the fact that $S$ is a free $A$-module
on the monomials.
\end{proof}

\begin{lem}\label{lem:Groupring}
Let $A$ be a commutative ring with unit and let 
\[
0 \to L \to M \to N \to 0
\]
be a short exact sequence of finitely generated abelian groups.
Let $I_L$ be the ideal of the group ring $A[M]$ generated by the
set $\{s-1 | s \in S\}$, where $S$ is a generating set of $L$.
Then the map of group rings $\frac{A[M]}{I_L} \to A[N]$ is an isomorphism.
\end{lem}
\begin{proof}
This is an elementary exercise and a proof can be found in
\cite[Proposition~2]{May}.
\end{proof}

\vskip .5cm 

{\bf{Proof of Theorem~\ref{thm:CTB}}:}
We shall prove this theorem along the same lines as the proof of
Theorem~\ref{thm:LHT}. Recall that our base field $k$ is perfect.
We consider the stratification of $B$ by smooth locally closed subschemes
as in ~\eqref{eqn:LHT-fil}. We shall follow the notations used in the
proof of Theorem~\ref{thm:LHT}. It suffices to show by induction on $i$
that the theorem is true when $B$ is replaced by each $U_i$.
Since $U_0 = \emptyset$ and since $E \xrightarrow{p} B$ is trivial over $U_1$,
the desired isomorphism for $i \le 1$ follows from 
Lemma~\ref{lem:Trvial-bundle-case}.

Given a smooth locally closed subscheme $j: U \inj B$, let
$\zeta^U_i = j^*(\zeta_{\chi'_i}) \in \bdK _*(U)$ for $1 \le i \le r$ and set
$J^G_U = \left(\chi'_1 - \zeta^U_{1}, \cdots ,
\chi'_r - \zeta^U_{r}\right)$.

We have seen in the proof of Lemma~\ref{lem:T-case} that for any such 
inclusion $U \subseteq B$,
$R_T(\bdK _*(U), \Delta)$ is same as 
$\bdK ^T_0(X) {\underset{\Z}\otimes}\bdK_*(U)$. 
Moreover, the maps
\begin{equation}\label{eqn:CTB12*}
R_G(\bdK _*(B), \Delta) {\underset{\bdK _*(B)}\otimes} \bdK _*(U)
\cong \frac{R_T(\bdK _*(B), \Delta)}{J^GR_T(\bdK _*(B), \Delta)} 
{\underset{\bdK _*(B)}\otimes} \bdK _*(U) \to 
\frac{R_T(\bdK _*(U), \Delta)}{J^G_UR_T(\bdK _*(U), \Delta)} 
\end{equation}
\[
\hspace*{10cm}
\to R_G(\bdK _*(U), \Delta)
\]
are all isomorphisms.

We now consider the diagram: 
\begin{equation}\label{eqn:LHT&CT}
\xymatrix@C.4pc{
R_G(\bdK _*(U_i), \Delta) \ar[r] \ar[d]_{\Phi^{U_i}_G} &
R_G(\bdK _*(V_{i+1}), \Delta) \ar[r] \ar[d]_{\Phi^{V_{i+1}}_G} &
R_G(\bdK _*(U_{i+1}), \Delta) \ar[r] \ar[d]_{\Phi^{U_{i+1}}_G} &
R_G(\bdK _*(U_i), \Delta) \ar[r] \ar[d]_{\Phi^{U_i}_G} &
R_G(\bdK _*(V_{i+1}), \Delta) \ar[d]_{\Phi^{V_{i+1}}_G} \\
\bdK _*(\mathfrak{X}_i) \ar[r] & \bdK _*(\mathfrak{W}_{i+1}) \ar[r] & 
\bdK _*(\mathfrak{X}_{i+1}) \ar[r] & 
\bdK _*(\mathfrak{X}_i) \ar[r] & \bdK _*(\mathfrak{W}_{i+1}).}
\end{equation}

Using ~\eqref{eqn:CTB12*}, we see that the top row of  ~\eqref{eqn:LHT&CT} 
is obtained by tensoring the localization exact sequence  
\[
\cdots \to \bdK _*(U_{i}) \to \bdK _*(V_{i+1}) \to \bdK _*(U_{i+1}) \to \bdK _*(U_{i}) \to
\bdK _*(V_{i+1}) \to \cdots
\]
of $\bdK _*(B)$-modules with $R_G(\bdK _*(B), \Delta)$. Hence, this row is exact by
Corollary~\ref{cor:T-case*}. The bottom row is anyway a localization exact 
sequence. 

We now show that the diagram ~\eqref{eqn:LHT&CT} commutes. 
It is clear that the third square commutes
and the fourth square is same as the first. So we need to check that the
first two squares commute.

Let $\alpha: V_{i+1} \inj U_{i+1}$ and $\beta:   \mathfrak{M}_{i+1} \inj 
\mathfrak{X}_{i+1}$ be the closed immersions of smooth schemes and stacks.
Following the notations in the proof of Theorem~\ref{thm:LHT}, 
we see that for any $u \in \bdK _*(U_i)$ and for any monomial 
$\gamma(\underline{t}) = t^{i_1}_1\cdots t^{i_d}_d$,

\begin{equation}\label{eqn:CTB13*}
\begin{array}{lll}
\delta \circ \Phi^{U_i}_G(u \otimes \gamma(\underline{t})) & = &
\delta\left(\pi^*_{U_i}(u) \cdot 
\eta^*_i\left([{E(\gamma)}/G]\right)\right) \\
& {=} & \delta(\pi^*_{U_i}(u)) \cdot 
\iota^*_{i+1}\left([{E(\gamma)}/G]\right) \\ 
& = & \pi^*_{V_{i+1}}(\delta(u)) \cdot 
\iota^*_{i+1}\left([{E(\gamma)}/G]\right) \\ 
& = & \Phi^{V_{i+1}}_G \left(\delta(u)\otimes \gamma(\underline{t})\right) \\
& = & \Phi^{V_{i+1}}_G \circ \delta(u \otimes \gamma(\underline{t})),
\end{array}
\end{equation}
where $E(\gamma) \in \bdK _*(\mathfrak{X})$ is as in ~\eqref{eqn:vanish4*}. 
The second equality follows from the Leibniz rule and the third equality
follows from the commutativity of ~\eqref{eqn:LHT&}.
This shows that the first (and the last) square commutes.

To show the commutativity of the second square, let $v \in \bdK _*(V_{i+1})$.
We then have
\begin{equation}\label{eqn:CTB13*}
\begin{array}{lll}
{\beta}_* \circ \Phi^{V_{i+1}}_G (v \otimes \gamma(\underline{t})) & = &
{\beta}_* \left(\pi^*_{V_{i+1}}(v) \cdot 
\iota^*_{i+1}\left([{E(\gamma)}/G]\right)\right) \\ 
& = &
{\beta}_* \left(\pi^*_{V_{i+1}}(v) \cdot 
\beta^* \circ \eta^*_{i+1}\left([{E(\gamma)}/G]\right)\right) \\ 
& = & {\beta}_*(\pi^*_{V_{i+1}}(v)) \cdot
\eta^*_{i+1}\left([{E(\gamma)}/G]\right) \\
& = & \pi^*_{U_{i+1}}(\alpha_*(v)) \cdot
\eta^*_{i+1}\left([{E(\gamma)}/G]\right) \\
& = & \Phi^{U_{i+1}}_G (\alpha_*(v) \otimes\gamma(\underline{t}) ) \\
& = & \Phi^{U_{i+1}}_G \circ \alpha_* (v \otimes \gamma(\underline{t})),
\end{array}
\end{equation}
where third equality follows from the projection formula and the
fourth equality follows from the commutativity of ~\eqref{eqn:LHT&}.
This shows that the second square commutes.

The first and the fourth vertical arrows are isomorphisms 
by induction. The second and the fifth vertical arrows are isomorphisms by 
Lemma~\ref{lem:Trvial-bundle-case}. 
Hence the middle vertical arrow is also an isomorphism by 5-lemma.
This concludes the proof of Theorem~\ref{thm:CTB}.
$\hspace*{12.5cm} \hfil\square$

\vskip .3cm 

\begin{remk}\label{remk:Non-reduced-bundle}
It was assumed in Theorem~\ref{thm:CTB} that $G$ is subgroup of $T$.
Since $\mathfrak{X}$ is just the toric stack $[{E(X)}/G]$ associated
to the data $(E(X), G \xrightarrow{\phi} T)$, the
general case can always be reduced to the case of Theorem~\ref{thm:CTB}.
We refer to Remark~\ref{remk:non-red} for how this can be done. 
\end{remk}

\vskip .3cm

\noindent\emph{Acknowledgments.}
Parts of this work were carried out while the first author was visiting the Tata
Institute of Fundamental Research, while the second author was visiting
the Mathematics department of Ohio state university,
Columbus and also while both the authors were visiting
the Mathematics department  of the Harish Chandra
Research Institute, Allahabad. The first author was also
supported by an adjunct professorship at the same institute. They would
 would like to thank these
departments for the invitation and financial support during these visits.
They also would like to thank Hsian-Hua Tseng for helpful comments on an 
earlier version of this paper.

\enlargethispage*{75pt}

\end{document}